\documentclass{amsart}
\usepackage{verbatim,amsfonts,color,mathrsfs,stmaryrd,graphicx, mathdots}
\usepackage{xcolor}
\usepackage{tikz}
\usepackage{tikz-cd}
\usepackage{tkz-euclide}
\usetikzlibrary{patterns}   
\usetikzlibrary{positioning}
\usetikzlibrary{decorations.pathreplacing,calligraphy}

\usepackage{pgfplots}

\pgfplotsset{compat=1.6}
\pgfplotsset{soldot/.style={color=blue,only marks,mark=*}} 
\pgfplotsset{holdot/.style={color=blue,fill=white,only marks,mark=*}}

\theoremstyle{plain}
\newtheorem{Thm}{Theorem}[section]
\newtheorem{Cor}[Thm]{Corollary}

\newtheorem{Lem}[Thm]{Lemma}
\newtheorem{Prop}[Thm]{Proposition}

\theoremstyle{definition}
\newtheorem{Def}[Thm]{Definition}
\newtheorem{Rmk}[Thm]{Remark}
\newtheorem{Eg}[Thm]{Example}

\theoremstyle{remark}

\usepackage{hyperref}

\begin{document}

%Topmatter 

\title[Proofs of ergodicity of piecewise M\"obius interval maps using planar extensions]{Proofs of ergodicity of piecewise M\"obius interval maps using planar extensions}
\author{Kariane Calta}
\address{Vassar College}  
\email{kacalta@vassar.edu }

\author{Cor Kraaikamp}
\address{Technische Universiteit Delft and Thomas Stieltjes Institute of Mathematics\\ EWI\\ Mekelweg 4\\ 2628 CD Delft, the Netherlands}
\email{c.kraaikamp@tudelft.nl}

\author{Thomas A. Schmidt}
\address{Oregon State University\\Corvallis, OR 97331}
\email{toms@math.orst.edu}
%\keywords{ergodicity, invariant measures, M\"obius maps, continued fractions}
\date{3 May 2024}

%End topmatter

\begin{abstract}  We give two  results for deducing dynamical properties of piecewise M\"obius interval maps from their related planar extensions.    First,    eventual expansivity and   the existence of an ergodic  invariant probability measure  equivalent to Lebesgue measure both follow from mild finiteness conditions on the planar extension along with a new property ``bounded non-full range" used to relax traditional Markov conditions.     Second, the ``quilting" operation to appropriately nearby planar systems, introduced by Kraaikamp and co-authors,  can be used to prove several key dynamical properties of a piecewise M\"obius interval map.   As a proof of concept, we apply these results to recover known results   on the well-studied Nakada $\alpha$-continued fractions; we obtain similar results for interval maps derived from an infinite family of non-commensurable Fuchsian groups.       
\end{abstract}

\maketitle

\tableofcontents

\section{ Introduction} For much of the technical vocabulary mentioned here, see \S~\ref{s:back} Background.

 \subsection{Historical overview} Metric number theory can be said to have begun with Gauss's discovery of the invariant measure for the regular continued fraction map.   The map, defined on $[0,1]$, fixes $x=0$ and  for all nonzero $x$ is given by $f: x \mapsto 1/x - \lfloor 1/x\rfloor$ where $\lfloor \cdot \rfloor$ denotes the floor or nearest integer function.    In a letter to Laplace in 1812, Gauss stated what in modern terms is that the measure given by $dx/(\ln 2\, (1+x)\,)$   is invariant under $f$.   He did not, however, state how he found this measure.    

Regular continued fractions appear in various settings in mathematics.  For instance, in 1924  E.~Artin  used them to show that the existence of unit tangent vectors of the modular surface whose orbit under the geodesic flow is dense in the unit tangent bundle.  
 In 1935, Hedlund used the connection to regular continued fractions to prove that the geodesic flow on the modular surface is ergodic with respect to the natural measure is ``metrically transitive", a property which implies ergodicity.  The following year,  E.~Hopf showed that the geodesic flow on the unit tangent bundle of any hyperbolic surface of finite volume is ergodic.   

 In fact, this last result can be used to show the ergodicity of  the invariant measure for the regular continued fraction map and even determine the measure beforehand.    
In works of Adler-Flatto in the 1980s and 1990s and of C.~Series in 1991,  it is shown that one can find a cross section for the geodesic flow on the unit tangent bundle of the modular surface ---  thus a subset of the unit tangent bundle  which every flow line meets transversely ---  and then show that the dynamical system of  the first return map to this cross section by the flow is an extension of the regular continued fraction system.    

In 1977, Nakada, Ito and Tanaka \cite{NakadaItoTanaka} gave a  more elementary presentation of an extension of  the regular continued fraction system.   They considered the map $\mathcal T$ on the unit square defined by  $\mathcal T(x,y) = (\,1/x - \lfloor 1/x\rfloor, 1/(y + \lfloor 1/x\rfloor)\,)$.   Since the values $\lfloor 1/x\rfloor$ are locally constant,  using the Jacobian determinant of $\mathcal T$  one easily shows that the measure $\mu$ given by $(1+xy)^{-2}\, dx\, dy$ is $\mathcal T$-invariant.     The {\em marginal measure} $\nu$ of this is the measure on $[0,1]$ given by assigning to any Borel set $E$ the value of $\mu$ on the subset of the square fibering over $E$.  Here one easily finds Gauss's measure, up to the normalizing constant.  Nakada, Ito and Tanaka showed that the system of $\mathcal T$    gives the natural extension of the regular continued fraction map.    (Keane \cite{Keane} has suggested that Gauss may have found his invariant measure by use of a closely related system.)
 
In 1981,  Nakada \cite{N} introduced his  $\alpha$-continued fractions, which form a one dimensional family of interval maps, $T_{\alpha}$  with  $\alpha \in [0,1]$.   (In fact,  $T_1$ is the Gauss continued fraction map,  and $T_{1/2}$ is the ``nearest-integer continued fraction" map.)   Using planar natural extensions, he gave the Kolmogorov--Sinai  measure theoretic entropy  --- hereafter simply \emph{entropy}, for those maps corresponding to $\alpha \in [1/2, 1]$.     In 1991,     Kraaikamp gave a more direct calculation of these entropy values by using his $S$-expansions,  based upon inducing past subsets of the planar natural extension of the regular continued fraction map given in \cite{NakadaItoTanaka}.      

Let $h(T_{\alpha})$ denote the  entropy  of $T_{\alpha}$.   In 2008 Nakada and Natsui ~\cite{NN} gave explicit intervals on which $\alpha \mapsto h(T_{\alpha})$ is respectively constant, increasing, decreasing.   
Indeed, they showed this by exhibiting intervals of $\alpha$ such that  $T^{k}_{\alpha}(\alpha) = T^{k'}_{\alpha}({\alpha-1})$ for pairs of positive integers $(k,k')$ --- such intervals are now known as {\em matching intervals}\label{matchingI} ---  and showed that  the function $\alpha \mapsto h(T_{\alpha})$ is constant (resp.\ increasing, decreasing) on such an interval if $k=k'$ (resp.\ $k>k'$, $k<k'$).       That same year,   Luzzi and Marmi  \cite{LM}  strongly suggested that    $\alpha \mapsto h(T_{\alpha})$ is a continuous function of $\alpha$.  They also asked if every $T_{\alpha}$ is the factor of some cross section to the geodesic flow on the unit tangent bundle of the modular surface.    The continuity was proven in 2012  by both C.~Carminati  and G.~Tiozzo \cite{CT} and  \cite{KraaikampSchmidtSteiner}.  This latter paper used explicit constructions of planar natural extensions.   A presumably necessary commonality of the two papers was the description of the complement of the set of all matching intervals, called the {\em exceptional set}.   The question of Luzzi-Marmi was answered affirmatively in \cite{ArnouxSchmidtCross}, using what they called ``Arnoux's transversal" to find the cross sections.  
   
Many generalizations of regular continued fractions have been studied.    Katok-Ugarcovici  \cite{KatokUgarcoviciStructure} introduced the family of $(a,b)$-continued fraction maps and determined the full subset of its two-dimensional parameter set for which matching occurs; see also    \cite{CIT,   KatokUgarcoviciApps}.   These continued fractions  are also associated to the modular group.    To each of the triangle Fuchsian groups known as the Hecke groups, \cite{DKS} associated a one-parameter family of continued fraction maps and began the study of their entropy functions; see also \cite{KraaikampSchmidtSmeets}. 

\subsection{Specific motivation, Two main results}   We call on the two main results of this paper in  \cite{CaltaKraaikampSchmidtPlanar}.
In \cite{CaltaKraaikampSchmidt}, we studied a one parameter family of piecewise M\"obius interval maps for each of a countably infinite number of triangle Fuchsian groups.    Although planar extensions are barely mentioned there, our paper was informed by numerous calculations of them.       As opposed to say the Nakada $\alpha$-continued fractions,  infinitely many of   the  maps considered in \cite{CaltaKraaikampSchmidt}  are {\em not expansive} maps.   A direct proof that each is {\em eventually expansive} seems tedious at best; this motivated us to seek   a general result that can be easily applied to deduce eventual expansivity.   We give such a result here as part of Theorem~\ref{t:evenExpanErgoNaturally}.    

One expects sufficiently nice continued fraction maps to be ergodic with respect to some measure which is absolutely continuous with respect to Lebesgue measure; the easiest setting to prove such results is when a Markov condition is fulfilled.   In the setting of \cite{CaltaKraaikampSchmidt}, and in many cases of continued fraction-like maps,  Markov properties do not hold.   In Definition~\ref{d:Def} below,  we introduce a property that is often fulfilled  in these settings.  That this property and basic finiteness conditions satisfied by a planar extension for a map then imply ergodicity and more is also given in Theorem~\ref{t:evenExpanErgoNaturally}.  As  an application, in \S~\ref{ass:egAlpOne} we show that each of an infinite collection of maps is ergodic.\\

We also study  a technique  used to date for solving for the planar extension of a piecewise M\"obius interval map beginning with such a planar extension for a sufficiently ``nearby" map.   This technique, called {\em quilting}, was introduced in \cite{KraaikampSchmidtSmeets},  and has its roots  in the discussion of the two-dimensional interpretation of ``insertion" and ``deletion" in the Ph.D. dissertation \cite{Kraaikamp}.    Theorem~\ref{t:finQuiltIsFine} shows that one can use quilting to prove that fundamental dynamical properties are 
shared between appropriately nearby systems.   We give applications of this in the setting of ``matching intervals" in \S~\ref{ss:FiniteQuiltingClose}.

One can thus pass from a system, say proven to have  properties by use of Theorem~\ref{t:evenExpanErgoNaturally}, to nearby systems and deduce that they also enjoy these properties.   In \S~\ref{s:AltNakErg} we show that this approach gives an alternate path to proving properties of the well-studied Nakada $\alpha$-continued fractions.    While pursuing this path, we discovered what seems to be an unnoticed symmetry within the planar natural extensions of these maps, see the introductory paragraph of \S~\ref{sss:RectangleAndFiberComplements} for more on this symmetry.

\phantom{here}\\  
\noindent
{\bf Convention}   Throughout, we will allow ourselves the minor abuse of using adjectives such as injective, surjective and bijective to mean in each case {\em up to measure zero}, and thus similarly where we speak of disjointness and the like we again will assume the meaning being taken to include the proviso ``up to measure zero" whenever reasonable.\\

\subsection{Outline} 
  In \S~\ref{s:back} we introduce basic terminology, notation and results from dynamical systems and ergodic theory;   review the settings for our examples and illustrative applications; and summarize further background material for Propositions~\ref{p:quiltFirstReturn} and ~\ref{p:quiltBernoulli}, which extend the second main result.  
 \S~\ref{s:BddNonFullRange} states and proves our first main result, Theorem~\ref{t:evenExpanErgoNaturally}, and gives an application.   \S~\ref{s:quilting} states and proves our second main result,  Theorem~\ref{t:finQuiltIsFine}, and related results.   \S~\ref{s:AltNakErg} gives an application of our results in the setting of the Nakada $\alpha$-continued fractions.   Whereas the use of our main results are straightforward, here the setting is admittedly technical.  We hope the reader will enjoy the rich details.

\subsection{Thanks}  It is a real pleasure to thank the referee for strongly recommending a clearer presentation, for mathematically helpful comments, and for additions to the bibliography.
 
\section{Background}\label{s:back}
  We collect standard background material in the  \S\S~\ref{ss:BasicsDynSys}, ~\ref{ss:planarBckgrd} and ~\ref{ss:ModGpSurfGeoFlo};  most of this material can be found in various textbooks, such as  \cite{DK, einWar, KatokBoek, KH, Petersen}.  For a different perspective on matters of \S~\ref{ss:planarBckgrd}, see various works of P.~Kůrka, such as the text \cite{KurkaBook}  the joint work \cite{KurkaKazda}.  In \S~\ref{ssHitoshi} we recall various results about Nakada's $\alpha$-continued fractions, both to illustrate the prior material and to use for motivation and application of the results of this paper.  With the same ends, we give brief summaries of  further results from the literature in the remaining portion of this section.

\subsection{Basics of dynamical systems}\label{ss:BasicsDynSys}   A {\em  dynamical system} is any  $(X, T, \mathscr B, \mu)$  where $X$ is a 
topological space $X$, $T:X\to X$ is a  function, $\mathscr B$ a sigma algebra, and $\mu$ a $T$-invariant measure on $\mathscr B$.  (In all that follows, we consider only Borel sigma algebras, unless stated otherwise.)  A dynamical system $(X, T, \mathscr B, \mu)$ is an {\em extension} of $(Y, S, \mathscr B', \nu)$ if there is a measurable map $\pi:X \to Y$ such that   there are sets of full measure $Y' \subset Y$ and $X'\subset X$ such that 
$S(Y') \subset Y'$ and $T(X') \subset X'$ and a measurable surjective map $\pi: X' \to Y'$ such that 
$\pi\circ T = S\circ \pi$ and $\mu \circ \pi^{-1} = \nu$.    We also say that the second system is a {\em factor} of the first.  The {\em natural extension} of a dynamical system was introduced by Rohlin, and defined by means of an inverse limit.  It is a minimal invertible extension in the sense that any invertible system which is an extension of $(Y, S, \mathscr B', \nu)$ is also an extension of it.   Naturally enough, the natural extension of a dynamical system is only well defined up to isomorphism; we will be most interested in planar extensions which give natural extensions, see \S~\ref{ss:planarBckgrd}.

The Kolmogorov--Sinai  measure theoretic entropy, which as stated above we refer to simply as entropy  and usually denote in the form $h(T)$ ---  is an invariant of a dynamic system, which roughly speaking measures its complexity.  In fact, Rohlin introduced the notion of the natural extension system to aid in the study of entropy  and  showed that the original system and its natural extension share entropy values.   

For a dynamical system $(X, T, \mathscr B, \mu)$ of finite measure and a subset $E \subset X$ of positive measure, the {\em induced transformation} on $E$ is  $T_E: E \to E$ given by $T_E(x) = T^k(x)$ where $k \in \mathbb N$ is minimal such that $T^k(x) \in E$.  (By the Poincar\'e Recurrence Theorem, the set of $x \in E$ such that there is some such $k$ has full measure in $E$, and in fact one defines $T_E$ to be the identity on the complement of this subset.)  We set $\mu_E$ to be the restriction to $E$ of $\mu$ scaled by $1/\mu(E)$.   This allows one to define a dynamical system for $T_E$.   The following is a key tool in the study of planar extensions. {\em Abramov's Formula}   states that the entropy of the induced system on $E$ is the quotient of the entropy of the original system divided by the measure of $E$, in short
\begin{equation}\label{e:AbramForm}
 h(T_E)= h(T)/\mu(E).
\end{equation}

One also has {\em Rohlin's Entropy Formula} for an interval map $T: \mathbb I \to \mathbb I$ which is ergodic with respect to an invariant probability measure $\nu$  and such that the derivative $T'$ exists $\nu$-almost everywhere, 
\begin{equation}\label{e:rohlinEnt}
h(T) = \int_{\mathbb I} \ln \vert\, T'(x)\,\vert\, d\nu.
\end{equation}

\subsection{Piecewise M\"obius maps, cylinders, planar extensions}\label{ss:planarBckgrd}   The group $\text{GL}(2, \mathbb R)$ of invertible integral $2 \times 2$ matrices acts by way of M\"obius transformations on the Riemann sphere, that is on the set of complex numbers union a point denoted $\infty$.  To wit, for  
\begin{equation}\label{e:justM}
M = \begin{pmatrix} a&b\\c&d\end{pmatrix}
\end{equation}
 in $\text{GL}(2, \mathbb R)$    and $z \in \mathbb C$, one has $M\cdot z = \frac{az+b}{cz+d}$.    We use mainly the restriction of this action to the real numbers.   Note that the action is projective in the sense that we can and do restrict to the case of $\det M = \pm 1$.     Those $M$ of determinant equal to $1$ form $\text{SL}(2, \mathbb R)$.  Taking the quotient  by its center $\{\pm \text{I}\}$, we obtain 
the group $\text{PSL}(2, \mathbb R)$.   

The following notation is perhaps slightly inelegant, but we will find it useful.    Let $\text{SL}^{\pm 1}_2(\mathbb R)$ denote the subgroup of $\text{GL}_2(\mathbb R)$ comprised of those elements whose determinant is $1$ or $-1$, and let $\text{PSL}^{\pm 1}_2(\mathbb R)$ be its quotient by $\{\pm I\}$.   Then,   $\text{PSL}^{\pm 1}_2(\mathbb R)$ contains $\text{PSL}_2(\mathbb R)$ as a subgroup of index two.  It is a standard abuse to represent an element $[M]  = \pm M \in \text{PSL}^{\pm 1}_{2}(\mathbb R)$ simply by $M \in \text{SL}^{\pm 1}_{2}(\mathbb R)$.

A {\em piecewise M\"obius interval map} is a function $T$ on a subinterval $\mathbb I \subset \mathbb R$ with values in $\mathbb I$ such that  
there is a partition $\mathbb I= \cup_{\beta \in \mathcal B}\, K_\beta$ with  $T(x) = M_{\beta}\cdot x$ for all $x \in K_\beta$.    We will assume that each $K_{\beta}$ is an interval and is taken as large as possible.   We call these $K_\beta$ the (rank one) {\em cylinders} for $T$.   Similarly, a cylinder of rank $m>1$ is the largest interval on which $T^m$ is given by the action of some $M_{\beta_m}\cdots M_{\beta_1}$.    We say that the  cylinder $K_{\beta}$ is {\em full} if $T(K_{\beta}) = \mathbb I$.  Naturally enough, any cylinder which is not full is called  {\em non-full}.

The standard number theoretic planar map associated to a M\"obius transformation $M$ is 
\[
\mathcal{T}_M(x,y) :=  \bigg( M\cdot x,  RMR^{-1}\cdot y\,\bigg)\, \quad  \text{for}\;\; x \in \mathbb R\setminus \{M^{-1}\cdot \infty\},\ y \in \mathbb R\setminus \{(RMR^{-1})^{-1}\cdot \infty\}\,,
\]
where 
\begin{equation}\label{e:justR}
R = \begin{pmatrix}0&-1\\1&0\end{pmatrix}.
\end{equation}
Thus,   $\mathcal{T}_M(x,y) = (\,M\cdot x, -1/(M\cdot (-1/y))\,)$.
 As mentioned in the introduction, an elementary Jacobian matrix calculation verifies that the measure  $\mu$ on $\mathbb R^2$ given by
\begin{equation}\label{e:muDefd}
d\mu = \dfrac{dx\, dy}{(1 + xy)^2}
\end{equation}
is (locally) $\mathcal{T}_M$-invariant.  

 For  a piecewise M\"obius interval map  $T$ we then set  
\begin{equation}\label{e:2DlocMap}
\mathcal{T}(x, y) = \mathcal{T}_{M_{\beta}}(x,y)    \quad  \text{whenever}\;\; x \in K_\beta,\ y \in \mathbb R\setminus \{N^{-1}\cdot \infty\}\,.
\end{equation}
Suppose that  $\Omega \subset \mathbb R^2$ projects onto the interval $\mathbb I$ and is a domain of bijectivity  of $\mathcal T$,  that is $\mathcal T$ is bijective on $\Omega$ up to $\mu$-measure zero.  Let $\mathcal B$ be the Borel algebra of $\Omega$, we then call the system $(\mathcal T, \Omega, \mathcal B, \mu)$  a {\em planar extension} for $T$.   We will occasionally abuse this terminology and say that $\Omega$ or $\mathcal T$ is the planar extension.  Similarly, we will make occasional use of the words   $T$ is of  {\em positive planar extension} to mean that $T$ has a planar extension with $0<\mu(\Omega)<\infty$.

We will have occasional need to use a planar extension of $T$ for which the invariant measure is Lebesgue measure.   Let 
\begin{equation} \label{eqZmap}
\mathcal Z(x,y) = (x, y/(1+xy))
\end{equation}
and  for each $M= \begin{pmatrix} a&b\\c&d\end{pmatrix}$ as above, let  $\widehat{\mathcal T}_M = \mathcal Z \circ \mathcal T_M\circ \mathcal Z^{-1}$.  Then 
\begin{equation} \label{eqLebMap}
\widehat{\mathcal T}_M(x,y) = ( M\cdot x, \, \det M  (\,(c x+d)^2 y - c (cx + d)\,)\,).
\end{equation} 
 We claim that Lebesgue measure is $\widehat{\mathcal T}_M$-invariant.  The derivative with respect to $x$ of $x \mapsto M\cdot x$ equals $\det M/(cx + d)^2$.  This also equals the multiplicative inverse of the  partial derivative with respect to $y$ of $\det M\,( (c x+d)^2 y - c (cx + d)\,)$.  An elementary Jacobian matrix calculation hence verifies that Lebesgue measure on $\mathbb R^2$ is (locally) $\widehat{\mathcal T}_M$-invariant.         See Figure~\ref{f:standVsLebesgue} for an indication of the effect of $\mathcal Z$.
 
 Suppose that some $T$ with its planar natural extension   $(\mathcal T, \Omega, \mathcal B, \mu)$ is given.  We let $\widehat{\mathcal{T}}(x, y) = \mathcal Z \circ \mathcal T \circ \mathcal Z^{-1}$, thus $\widehat{\mathcal{T}}$  is given piecewise by various   $\widehat{\mathcal T}_M$ on the domain of bijectivity $\Sigma = \mathcal Z(\Omega)$.   We call the corresponding dynamical system the {\em Lebesgue planar extension} of $T$.   See Figure~\ref{f:standVsLebesgue} for a view of both types of planar extensions for a particular map.
 
 \subsection{Nakada's $\alpha$-continued fractions}\label{ssHitoshi}  The Nakada $\alpha$-continued fractions form a one-parameter family of piecewise M\"obius interval maps. 
 
  For $\alpha \in [0,1]$, we let $\mathbb{I}_\alpha := [{\alpha-1}, \alpha]$. Then {\em Nakada's $\alpha$-continued fraction} map is defined as  $T_{\alpha}:\, \mathbb{I}_\alpha \to [{\alpha-1}, \alpha)$ by
\[
T_{\alpha}(x) := \bigg| \frac{1}{x} \bigg| -  \bigg\lfloor\,  \bigg| \frac{1}{x} \bigg| + 1 -\alpha \bigg\rfloor \qquad (x \neq 0),
\]
$T_{\alpha}(0) :=0$.
For $x \in \mathbb{I}_\alpha$, put
\[
\varepsilon(x) := \left\{\begin{array}{cl}+1 & \mbox{if}\ x \ge 0\,, \\ -1 & \mbox{if}\ x<0\,,\end{array}\right. \quad \mbox{and} \quad d_{\alpha}(x) := \bigg\lfloor \bigg| \frac{1}{x} \bigg| + 1 - \alpha \bigg\rfloor,
\]
with $d_\alpha(0) = \infty$.

    The cylinder $\Delta_{\alpha}(\varepsilon, d)$ is the set of $x$ such that   $(\varepsilon(x), d_{\alpha}(x)\,) = (\varepsilon, d)$.   Let  
    \[M_{(\varepsilon:d)} = \begin{pmatrix} -d& \varepsilon\\1&0  \end{pmatrix},\]
so  that $T_{\alpha}(x) = M_{(\varepsilon:d)} \cdot x$ on $\Delta_{\alpha}(\varepsilon, d)$.   (We will usually ignore the exceptional cylinder $\Delta_{\alpha}(+1, \infty)$ which contains only $x=0$.)
 Note that the only endpoints of any cylinder which $T_{\alpha}$  could possibly send to an interior point of $\mathbb I_{\alpha}$ are $\alpha-1$ or $\alpha$.   Thus, of the  infinitely many cylinders at most two cylinders, those of $\alpha-1$ or $\alpha$,  can be non-full. Furthermore, the  fullness or non-fullness of each of these two depends   only on the image of $\alpha-1$ or $\alpha$, respectively.   See Figure~\ref{f:nakadaCfPt39}.

 %---------------------------------Figure alpha=0.39 Plot T_{\alpha}  ----------------------------------------------
\begin{figure}[h]
\scalebox{.6}{
{\includegraphics{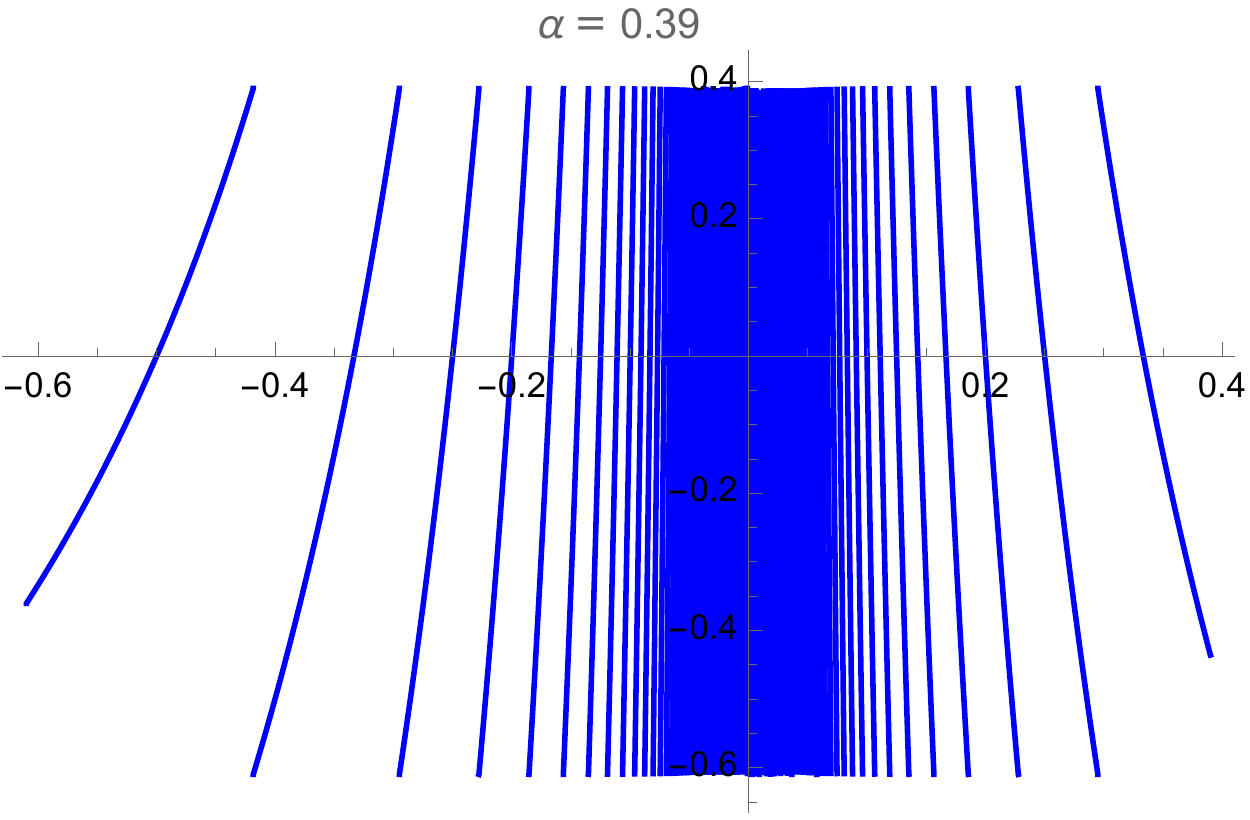}}}
\caption{Approximate graph of Nakada's $T_{0.39}$.  (Aspect ratio of the plot unequal to 1 for aesthetic reasons only.)  
Each cylinder corresponds to a branch of the graph. Notice that the two extreme cylinders are not full;  for each, this is due to exactly one endpoint of the cylinder having image in the interior of the interval of definition.  See Figure~\ref{f:NakadaCFpt39NatExt} for the planar extension of this function.}
\label{f:nakadaCfPt39}
\end{figure}
%-----------------------------------------------------------------------------------------------------------------------------------------

Furthermore, let
\[
\varepsilon_n = \varepsilon_{\alpha,n}(x) := \varepsilon(T^{n-1}_\alpha (x)) \quad \mbox{and}\quad d_n = d_{\alpha,n}(x) := d_\alpha(T^{n-1}_\alpha (x)) \qquad (n \ge 1).
\]
This yields the  $\alpha$-continued fraction expansion of $x \in \mathbb{R}$\,:
\[
x = d_0+ \dfrac{\varepsilon_1}{d_1 + \dfrac{\varepsilon_2}{d_2+\cdots}}\,, 
\]
where $d_0 \in \mathbb{Z}$ is such that $x-d_0 \in [{\alpha-1}, \alpha)$.
Note that when $\alpha=0$ this recovers the regular continued fractions, as mentioned in the introduction.   The collection of all finite words in the $(\varepsilon:d)$ which arise in the expansions of any $x \in \mathbb I_{\alpha}$ form the language $\mathcal L_{\alpha}$ for $T_{\alpha}$; any word in $\mathcal L_{\alpha}$ is called $\alpha$-{\em admissible}.  

 Also as indicated in the introduction,  a matching interval of parameter $\alpha$ values is an interval such that $T^{k}_{\alpha}(\alpha) = T^{k'}_{\alpha}({\alpha-1})$ for pairs of positive integers $(k,k')$ for all $\alpha$ in the interval.  Both \cite{CT} and  \cite{KraaikampSchmidtSteiner}  established that the complement in $[0,1]$ of the union of the matching intervals is a set of measure zero.  This complement is the {\em exceptional set}, denoted $\mathcal E$. 
 
For $\varepsilon$ and $d$ as above, let  $N_{(\varepsilon:d)} = \begin{pmatrix} 0& 1\\ \varepsilon &d  \end{pmatrix}$.     Note that  projectively,  $N_{(\varepsilon:d)} = (M_{(\varepsilon:d)}^{-1})^t = R M_{(\varepsilon:d)} R^{-1}$.  Define $\mathcal T_{(\varepsilon:d)}$ to be the map $(x,y) \mapsto (M_{(\varepsilon:d)}\cdot x, N_{(\varepsilon:d)}\cdot y)$.   Thus for $x \in \Delta_{\alpha}(\varepsilon:d)$ and any $y$ we have  $\mathcal T_{\alpha}(x,y) =\mathcal T_{(\varepsilon:d)}(x,y)$ in accordance with \eqref{e:2DlocMap}.

\subsection{Modular group and surface, geodesic flow, Fuchsian groups}\label{ss:ModGpSurfGeoFlo}
The upper half-plane is $\mathbb H= \{z=x+iy \,\vert\, x,y \in \mathbb R, y>0\}$ as a subset of the complex numbers.   The M\"obius action of  the group $\text{SL}(2, \mathbb R)$ preserves $\mathbb H$.  Indeed its elements act as isometries when we place the hyperbolic metric on $\mathbb H$, whose element of arclength squares to be $ds^2 = (dx^2 + dy^2)/y^2$.     The geodesics of $\mathbb H$ are either vertical lines and semi-circles, whose naive extensions meet the boundary real line perpendicularly.

For simplicity, let us say that a {\em unit tangent vector} on $\mathbb H$ is $u = (z, \theta)$ where $z\in \mathbb H$ and $\theta$ denotes a direction.   There is a unique hyperbolic geodesic passing through $z$ with tangent line of the given direction.  The collection of all of the unit tangent vectors is called the {\em unit tangent bundle}, $T^1 \mathbb H$.   One can extend the action of $\text{SL}_2(\mathbb R)$ on $\mathbb H$ to an action on $T^1 \mathbb H$.   The action is transitive and the stabilizer of the vertical unit tangent vector of basepoint $z=i$ is $\pm I$;  this  allows one to identify $T^1 \mathbb H$ with $\text{PSL}_2(\mathbb R)$.    

The {\em geodesic flow} on the unit tangent bundle is an action of the real numbers: given a nonnegative real number $t$ and a unit tangent vector $u$, the unit tangent vector   $t\cdot u$ is obtained by following the unique geodesic passing through $u$ in the positive direction for arclength $t$ and taking the unit tangent vector to this oriented geodesic at the new basepoint.  If $t<0$, we follow the geodesic in the opposite direction.    In terms of $\text{PSL}_2(\mathbb R)$, the geodesic flow for time $t$  is given by  sending $A\in \text{PSL}_2(\mathbb R)$ to $A g_t$,   where $g_t = \begin{pmatrix} e^{t/2}&0\\0&e^{-t/2}\end{pmatrix}$.

  A {\em Fuchsian group} $\Gamma$ is a discrete (with respect to the natural topology) subgroup of $\text{PSL}(2, \mathbb R)$ and in particular, it is a subgroup acting properly discontinuously on $\mathbb H$;  the quotient $\Gamma\backslash \mathbb H$ is a surface (or orbifold) and inherits a hyperbolic metric.  
  
  The quotient $\Gamma\backslash \mathbb H$ has a unit tangent bundle, given by equivalence classes (thus $\Gamma$-orbits) of unit tangent vectors of $\mathbb H$.  The geodesic flow descends so that for $t\in \mathbb R$,    $[A] \in \Gamma\backslash \text{PSL}_2(\mathbb R)$ is sent to $[A g_t]$, where the square brackets here denote $\Gamma$-cosets represented by the given elements.

  The  {\em modular group} is $\text{PSL}(2, \mathbb Z)$.   The {\em modular surface} is $\Gamma\backslash \mathbb H$ with $\Gamma = \text{PSL}(2, \mathbb Z)$.   The modular group is in particular a {\em triangle Fuchsian group}, of finite covolume. 
The {\em signature} of this Fuchsian group is $(0; 2,3, \infty)$, indicating that the modular surface is of genus zero and has quotient singularities of orders 2 and 3, and has a cusp: the modular surface is  a punctured sphere with the puncture being at infinite hyperbolic distance from the other points.

\subsection{Arnoux's method for cross sections to a geodesic flow}\label{ss:ArnouxMethod}
The following material is called upon in    \S~\ref{ss:quiltRealFirstReturn}.   
  
\subsubsection{Cross sections to a measurable flow,  Arnoux's transversal}\label{sss:crossFlowArnouxTrans}        
Let $(X, \mathscr B, \mu)$ be a measure space and $\Phi_t$ a measure preserving flow on $X$,  that is $\Phi:  X \times \mathbb R \to X$ is a measurable function such that for $\Phi_t(x) = \Phi(x,t)$, $\Phi_{s+t} = \Phi_s \circ \Phi_t$.   Then $\Sigma \subset X$ is a  {\em measurable cross section} for  the flow $\Phi_t$ if:     (1)  the flow orbit of almost every point meets $\Sigma$; (2) for almost every $x\in X$ the set of times $t$ such that $\Phi_t(x) \in \Sigma$ is a discrete subset of $\mathbb R$; (3)   for every Borel subset (in the subspace topology)   $A \subset \Sigma$ and for every  $\tau> 0$, {\em flow box}  $A_{[0,\tau]}\ :=  \{ \Phi_t(A)\, \mid \, A \in \Sigma,  t \in [0, \tau]\,\}$ is $\mu$-measurable.  

The {\em return-time function} $r= r_\Sigma$ is $r(x) =\inf\{t>0:\Phi_t(x)\in\Sigma\}$ and the {\em return map} $R:\Sigma\to\Sigma$ is defined by $R(x)= \Phi_{r(x)}(x)$. 
The {\em induced measure} $\mu_\Sigma$ on $\Sigma$ is defined from flow boxes: one sets 
$\mu_\Sigma(A)= \frac{1}{\tau} \mu(A_{[0,\tau]})$ for any  $0<\tau<\inf_{x\in A}\{r(x)\}$.

\medskip

\noindent
{\bf Convention}:  In all that follows,  we write {\em cross section} to denote measurable cross-section.
 \smallskip

A flow $\Phi_t$ is {\em ergodic} if for any invariant set either it or its complement is of measure zero.    
  The result of Hopf mentioned in the introduction states that if a Fuchsian group $\Gamma$ is of finite covolume, then the geodesic flow is ergodic with respect to the natural measure on the unit tangent bundle of $\Gamma\backslash \mathbb H$; see his reprisal in \cite{Hopf}.    A flow is {\em recurrent} if the $\Phi$-orbit of almost every point meets any positive measure set infinitely often.  By the Poincar\'e Recurrence Theorem, an ergodic flow on a finite measure space is recurrent.    Given a cross-section, a first return-time transformation  $(\Sigma, \mathscr B_\Sigma, \mu_\Sigma, R_\Sigma)$ is ergodic  whenever the flow is.    
  
 There is a natural measure on the unit tangent bundle $T^1 \mathbb H$:  the  {\em Liouville measure} is given  as the product of the hyperbolic area measure on $\mathbb H$ with the length measure on the circle of unit vectors at any point.   Liouville measure is (left- and right-) $\text{SL}_2(\mathbb R)$-invariant,  and  thus gives Haar measure on $G$.   In particular, this measure is invariant for the geodesic flow.     With standard normalizations, Liouville measure agrees with the Riemannian volume form.  These both descend modulo any Fuchsian group $\Gamma$.   For example, the volume of the unit tangent bundle of the modular surface is $\pi^2/3$.

  Arnoux, see say \cite{ArnouxSchmidtCross},  found an elementary manner to map Lebesgue planar extensions into the unit tangent bundle of appropriate surfaces so as to find cross sections to the geodesic flow.   This often allows one to express   an original interval map's system as a factor.  For $(x,y) \in \mathbb R^2$, let 
\[ A(x,y) =   \begin{pmatrix} x&xy-1\\1&y\end{pmatrix}.\]  
  {\em Arnoux's transversal}, $\mathscr A$, is the projection to $\text{PSL}_2(\mathbb R)$ of the set of all  $A(x,y)$. 
  The complement to   $\{[A g_t] \,\vert\, A\in \mathscr A, t \in \mathbb R\}$  is a null set for  Liouville measure.  Furthermore Liouville measure restricts to   $\mathscr A$ to be a constant multiple of $dx\,dy\, dt$.     

 Given $M\in \text{SL}^{\pm 1}_2(\mathbb R)$ of the form  \eqref{e:justM}, let  $\tau(M, x) = 2 \ln \vert cx + d\vert$ for any $x \neq -d/c$ (if $c \neq 0$).  The function $\tau(M, x)$ descends to be independent of representative $M$ for an element of $\text{PSL}^{\pm 1}_2(\mathbb R)$.    In what follows, we choose the representative $M$  such that $\vert cx + d\vert =   cx + d$.    For ease of legibility, set $t_0 =  \tau(M,x)$, $A = A(x,y)$ and $A' = A(\,\widehat{\mathcal T}_M(x,y)\,)$.    Then 
\begin{equation}\label{e:arnouxFlow}  M A  g_{t_0} = \begin{cases} A'&\text{if}\; \det M = 1;\\
\\
                                                         A' \;U &\text{if}\; \det M = -1,
                                  \end{cases}
\end{equation}                                 
where  $U = \begin{pmatrix} 1&0\\0&-1\end{pmatrix}$.

  Given a Fuchsian group $\Gamma$, we have the function
\[
\begin{aligned}
 \mathscr P =  \mathscr P_{\Gamma}: \mathbb R^2 &\to \Gamma\backslash\text{PSL}(2, \mathbb R)  \\
               (x,y) &\mapsto  [A(x,y)],
\end{aligned}               
 \]
 where again square brackets denote cosets.   Note that this is measure preserving when we use Lebesgue measure on $\mathbb R^2$ and the measure given by $dx\,dy$ inherited by  the projection $ \mathscr A_{\Gamma}$, of $\mathscr A$ to $\Gamma\backslash\text{PSL}(2, \mathbb R)$.   When we use $\mathscr P$ we will often commit the abuse of writing $A$ to represent the corresponding coset.

\subsubsection{Arnoux's method in the determinant one setting}    
  Now suppose that $T$ is a piecewise M\"obius interval map.  
 We say that the {\em group associated to $T$} is the group $\widehat{\Gamma}_T$ generated by the M\"obius transformations of $T$; thus,    $\widehat{\Gamma}_T = \langle M_\beta, \beta \in \mathcal B  \rangle \subset \text{PSL}^{\pm 1}_2(\mathbb R)$.  We let $\Gamma_T = \widehat{\Gamma}_T \cap \text{PSL}_2(\mathbb R)$, this is a subgroup of index $w_T \in \{1,2\}$.   In particular,  if $\det M_{\beta}=1$ for all $\beta$, then $w_T= 1$; we call this the {\em determinant one setting}.      
 
 If $T$ has a positive planar extension, we conjugate via $\mathcal Z$ to the Lebesgue planar extension of $T$.  We then apply the measure preserving $\mathscr P$, with $\Gamma = \Gamma_T$, to $\mathcal Z (\Omega)$.  When $\Gamma_T$ is a Fuchsian group, this is a subset of $T^1(\Gamma_T\backslash \mathbb H)$;  using the fact that any Fuchsian group has countably many elements,  in \cite{ArnouxSchmidtCommCF} it is shown that $\mathscr P$ is injective up to measure zero on $\mathcal Z (\Omega)$. 
 
 Arnoux's method in the determinant one setting is illustrated by the following result.   When  $T$ is a piecewise M\"obius interval map and $x \in \mathbb I$, we let $\tau(x) = \tau_T(x) = \tau(M, x)$ where $T(x) = M\cdot x$.    The result here combines (\cite{ArnouxSchmidtCommCF} Theorem~5.4, Corollary~1, and Proposition~4). 
%------------------------------------------------------------------------------------------ 
\begin{Thm}\label{t:detOneSetting} [Arnoux's Method] Let $T$ be a piecewise M\"obius interval map with positive planar extension, 
that each of M\"obius transformations giving $T$ is of determinant one, so   $\widehat{\Gamma}_T= \Gamma_T$ and that this is Fuchsian group  of finite covolume.   Then 
\[ \Sigma = \mathscr P_{\Gamma}(\,\mathcal Z (\Omega)\,)\]
is a cross section to the geodesic flow on $T^1(\Gamma_T\backslash \mathbb H)$.  Furthermore,  %with $g_{t_0}$ as in \eqref{e:arnouxFlow}, 
the system defined by 
\[\begin{aligned}
\phi: \Sigma \;\;\;&\to \;\Sigma\\
  {}      [A(x,y)] &\mapsto [M A(x,y) g_{\tau(x)}],
 \end{aligned}
 \]
 with $M$ such that $T(x) = M\cdot x$, 
 is an extension of $T: \mathbb I \to \mathbb I$.    Moreover, $\phi$ agrees with the first return map of the geodesic flow to $\Sigma$  if and only if $T$ is ergodic, eventually expansive,  and with entropy satisfying $h(T) \mu(\Omega_f) =  \emph{vol}(\,   T^1( \Gamma_T\backslash \mathbb H)\, )$.  When this holds, the   first return to $\Sigma$ gives a natural extension to $T$. 
\end{Thm}
%------------------------------------------------------------------------------------------ 
The initial statement of the theorem is shown by considering  \eqref{e:arnouxFlow} with the projection.    The subset $\{[A(x,y) g_t] \,\vert\, A(x,y) \in \Sigma, 0\le t\le \tau(x)\} \subset T^1(\Gamma_T\backslash \mathbb H)$ is invariant under the geodesic flow;  Hopf's result implies that this is all of $T^1(\Gamma_T\backslash \mathbb H)$ up to measure zero.       If $\phi$ agrees with the first return map,  then the map to $\mathscr A$ is such that the flow is expansive in the $x$-direction and contracting in the $y$-direction ---  recall that the map of \eqref{eqLebMap} preserves Lebesque measure  ---  and one can argue as in \cite{ArnouxSchmidtCommCF} that the first return system is indeed the natural extension.   The ergodicity of the flow implies that $\phi$ is ergodic, it then follows that $\widehat{\mathcal T}, \mathcal T$ and hence $T$ itself are ergodic.   The veracity of the equation involving the entropy is also in  \cite{ArnouxSchmidtCommCF}; in brief,  $\tau(x)$ is simultaneously  the arclength of the geodesic path following the flow line from $ [A(x,y)]$ to its image under $\phi$ and the (piecewise form of the) integrand in Rohlin's formula \eqref{e:rohlinEnt}. Recall that $\nu$ is the marginal measure of $\mu$ on $\Omega$, and that both $\mathcal Z$ and $\mathcal P_{\Gamma}$ are measure preserving.\\ 

In fact, even if $\phi$ is not given by the first return map, if $T$ is ergodic we can use Rohlin's formula and find that 
$h(T) \mu(\Omega_f)  \ge  \text{vol}(\,   T^1( \Gamma_T\backslash \mathbb H)\, )$.     Strict inequality holds exactly when there is a positive measure set of points in $\Sigma$ whose flow paths return to first agree with $\phi$ only  for some  $n^{\text{th}}$ return with $n>1$.

\subsubsection{Arnoux's method in the mixed determinant setting} 
One can extend the above results to the setting where not all $M= M_{\beta}$ are of determinant one. 

Let $\Sigma_+ = \mathscr P_{\Gamma}(\,\mathcal Z (\Omega)\,)$, thus equalling $\Sigma$ as above.    Fix $D\in \text{SL}^{\pm 1}_2(\mathbb R)$ of determinant $-1$, then let $\Sigma_{-} = \{[DAU] \,\vert\, A = A(x,y) \in \mathscr A \,\text{with} \,(x,y) \in \mathcal Z(\Omega)\}$.     Since Haar measure is both left- and right-multiplication invariant, we may and do assume that $\Sigma_{-}$ has the Lebesgue measure in terms of $x, y$.    

As usual, assume that $T(x) = M\cdot x$.  We define maps  under the restriction that $\det M = 1$
\[\begin{aligned}
    \alpha_M: \Sigma_+ &\to \Sigma_+\phantom{pleasealign}\text{and} \;\;&\beta_M: \Sigma_{-} &\to \Sigma_{-}\\
     [A(x,y)] &\mapsto [M A(x,y) g_{\tau(x)}]\;\;\;& [DA(x,y)U] &\mapsto [DMA(x,y) g_{\tau(x)}U]\,.
 \end{aligned}
 \]                     
When   $\det M = -1$ we define maps  
\[\begin{aligned}
    \gamma_M: \Sigma_{+} &\to \Sigma_-\phantom{pleasealign}\text{and} \;\;&\delta_M: \Sigma_{-} &\to \Sigma_{+}\\
     [A(x,y)] &\mapsto [DMA(x,y) g_{\tau(x)}]\;\;\;& [DA(x,y)U] &\mapsto [MA(x,y)U g_{\tau(x)}]\,.
 \end{aligned}
 \]                     
That these maps then do take values in the set indicated follows from considering  \eqref{e:arnouxFlow} with the projection $\mathscr P_{\Gamma}$, and the facts  that the diagonal matrices  $U$ and $g_t$ commute and finally that  $U^2$ is the identity.   Note that $DMA(x,y) g_{\tau(x)}U =   DM D^{-1} \, DA(x,y)U\, g_{\tau(x)}$, thus $\beta_M$ is given by left multiplication by a determinant one matrix and a right multiplication by a determinant one diagonal matrix.  Similarly,  the left multiplying matrices for the maps $\gamma_M$ and $\delta_M$ are  $DM$ and $MD^{-1}$, respectively.

%------------------------------------------------------------------------------------------ 
\begin{Thm}\label{t:detPlusMinOneSetting}[Arnoux's Method, 2] Let $T$ be a piecewise M\"obius interval map with positive planar extension,    with  $\Gamma_T$ a Fuchsian group  of finite covolume and $\Gamma_T \neq\widehat{\Gamma}_T$\,.  Suppose  $D \in \text{PSL}^{\pm 1}(\mathbb Z)$ is such  that (1)  $\Sigma_{+} \cap \Sigma_{-}$ is a null set;  (2)  $\forall M\in \Gamma_T$   also $DM D^{-1}\in \Gamma_T$; and, (3) $\forall M\in \widehat{\Gamma}_T \setminus \Gamma_T$ one has $DM\in\Gamma_T$.   
Then $\Sigma = \Sigma_{+} \cup \Sigma_{-}$ 
is a cross section to the geodesic flow on $T^1(\Gamma_T\backslash \mathbb H)$.  Furthermore,  %with $g_{t_0}$ as in \eqref{e:arnouxFlow}, 
the system defined by 
\[\begin{aligned}
\psi: \Sigma  \;\;\;&\to \;\Sigma \\
          \sigma = \sigma(x,y) &\mapsto  \begin{cases}\alpha_M(\sigma)&\text{if}\;  \sigma \in \Sigma_{+}\, \text{and } \det M = 1;\\
                                                                                   \beta_M(\sigma)&\text{if}\;  \sigma \in \Sigma_{-}\, \text{and } \det M = 1;\\
                                                                                     \delta_M(\sigma)&\text{if}\;  \sigma \in \Sigma_{+}\, \text{and } \det M = -1;\\                                                                                                                                                                              
                                                                                     \gamma_M(\sigma)&\text{if}\;  \sigma \in \Sigma_{-}\, \text{and } \det M = -1,\\
          \end{cases}
 \end{aligned}
 \]
 with $M$ such that $T(x) = M\cdot x$, 
 is an extension of $T: \mathbb I \to \mathbb I$.    Moreover,  $\psi$ agrees with the first return map of the geodesic flow to $\Sigma$ if and only if $T$ is ergodic, eventually expansive,  and has entropy satisfying $2 h(T) \mu(\Omega_T) =  \emph{vol}(\,   T^1( \Gamma_T\backslash \mathbb H)\, )$.  When this holds, the  planar system $\mathcal T$ on $\Omega_T$ gives a natural extension to $T$. 

\end{Thm}
%------------------------------------------------------------------------------------------ 

The following is the  main result of \cite{ArnouxSchmidtCross}. 
%------------------------------------------------------------------------ 
 \begin{Thm}\label{t:nakadaSection}    For any $\alpha \in (0,1]$ the Nakada $\alpha$-continued fraction map $T_\alpha$ is a factor of a first return system of the geodesic flow on the unit tangent bundle of the modular surface.  
 \end{Thm}
%------------------------------------------------------------------------  
\begin{proof}[Sketch]
Recall that the general form of the Nakada $\alpha$-continued fraction maps is $T_{\alpha}(x) = \begin{pmatrix}- d &\varepsilon\\1&0 \end{pmatrix}\cdot x$.   As $\begin{pmatrix}- d &\varepsilon\\1&0 \end{pmatrix} = \begin{pmatrix}- d &-1\\1&0 \end{pmatrix} \begin{pmatrix}1 &0\\0&-\varepsilon \end{pmatrix}$,    one can show that for all $\alpha$, the group $\Gamma_{T_{\alpha}}$ is the modular group.  As well, for $\alpha >0$,  the nontrivial coset of $\Gamma_{T_{\alpha}}$ in $\widehat{\Gamma}_{T_{\alpha}}$ is represented by $U$.   
 
   Let $D = R U = \begin{pmatrix}0&1\\1&0\end{pmatrix}$.  One has  $D A(x,y) U =  \begin{pmatrix}1&-y\\x&1-xy\end{pmatrix}$.   For each  $\alpha \in (0,1)$,   \cite{KraaikampSchmidtSteiner} give a positive planar extension $\Omega_{\alpha}$ for $T_{\alpha}$.  One now easily verifies that $\Sigma = \Sigma_{\alpha}$ satisfies the hypotheses of Theorem~\ref{t:detPlusMinOneSetting} to form a cross section to the geodesic flow on the unit tangent bundle of the modular surface.   By Arnoux's \cite{ArnouxCodage} the result holds in the case of $\alpha=1$, the case of regular continued fractions.  By  \cite{KraaikampSchmidtSteiner} (or by \cite{CT}), the product 
$h(T_{\alpha}) \mu(\Omega_{\alpha})$ is constant for the set of $0< \alpha \le 1$. Thus, for each of these values of $\alpha$, Theorem~\ref{t:detPlusMinOneSetting} applies to show that the {\em first} return system to $\Sigma_{\alpha}$ is an extension for the system of $T_{\alpha}$.
\end{proof} 
  
In fact, \cite{ArnouxSchmidtCross}, influenced by the formulation in the $\alpha=1$ case due to the geometric approach of  \cite{ArnouxCodage},  work with the cross section obtained by applying $R$ to $\Sigma$.

%------------------------------------------------------------------------ 
 
 \subsubsection{Realizable first return type}\label{sssRealFirst} The above results motivated 
\cite{ArnouxSchmidtCommCF} (in the determinant one setting)  to define a piecewise M\"obius interval map $T$ to be of {\em  first return type} if: (1) $T$ has a planar extension, with $0< \mu(\Omega) <\infty$; (2) $\Gamma_T$ is Fuchsian, of finite covolume; (3) $\tau_T(x) \ge 0$ for $\nu$-a.e.~ $x$; and, (4) for almost every  $(x,y) \in \Omega$ and every non-trivial  $N \in \widehat{\Gamma}_T$ with $ \mathcal T_N(x,y) \in \Omega$ and $\tau(N, x)\ge 0$, we have $\tau_T(x) \le  \tau(N, x)$.

Thus, in the determinant one setting if $T$ is of first return type then one can verify that the map $\phi$ of Theorem~\ref{t:detOneSetting} does accord with the first return map of the geodesic flow.  A natural extension is thus found for $T$, along with information about its entropy.   

We say that $T$ is of {\em realizable first return type} whenever  $T$ is of first return type and if  
 $\widetilde{\Gamma}_T \neq \Gamma_T$ then  also:  (5)  there is a  $D \in \text{PSL}^{\pm 1}(\mathbb Z)$  such  the hypotheses (1)--(3) of Theorem~\ref{t:detPlusMinOneSetting} hold. 
 
 Note that when $T$ is of {\em realizable first return type} then one of Theorem~\ref{t:detOneSetting} or Theorem~\ref{t:detPlusMinOneSetting} holds.  In particular, the system of $T$ is a factor of the first return map to a cross section for the geodesic flow on the unit tangent bundle of the surface uniformized by $\Gamma_T$. 
 
 \subsection{Bernoulli systems}\label{ssBern}   The following material is called upon in  \S~\ref{ss:quiltRealFirstReturn}.

A symbolic Bernoulli system is given by taking the shift map $\sigma$ on the bi-infinite sequences $ X = \mathcal A^{\mathbb Z}$ with $\mathcal A = \{ 1, 2, \dots, n\}$ for some $n \in \mathbb N$.   There is a standard manner to define the distance between two sequences, making $X$ into a (compact) metric space.   Cylinders of rank $m$ are defined as the set of sequences which agree for a specified consecutive sequence of length $m$.   For each choice of probability vector $(p_1, \dots, p_n)$, thus with $p_1+ \cdots + p_n = 1$, one gives a cylinder of rank $m$ the measure defined by taking the product of the $p_i$ corresponding to the `letters' $a_i \in \mathcal A$ which define the cylinder.   A limiting process results in a dynamical system $(X, \sigma, \mathscr B, \mu)$.     Any dynamical system isomorphic to such a system is called a {\em Bernoulli system} (alternatively, a Bernoulli scheme).     
Note that one can also consider  one-sided  Bernoulli shifts, but we ignore that here. 

Ornstein, both singly and with co-authors, established several celebrated results about Bernoulli systems.  
%------------------------------------------------------------------------------------------ 
\begin{Thm}\label{t:Ornstein} [Ornstein 1970, \cite{OrnsteinBSentropy}]  Entropy is a complete invariant for Bernoulli systems.  That is, if $(X, \mathcal T, \mathscr B, \mu)$ and $(Y,  U, \mathscr C, \nu)$ are two Bernoulli systems, then these systems are isomorphic if and only if they have the same entropy value.
\end{Thm}
%------------------------------------------------------------------------------------------ 

 A measure preserving flow  $\Phi_t$ (as in \S~\ref{sss:crossFlowArnouxTrans}) is called a {\em Bernoulli flow} if for each $t$, the system defined by $\Phi_t$ is  a Bernoulli system.     Due to a result of Abramov \cite{AbramovFlow}, it is traditional to say that the entropy value of a measure preserving flow   is the entropy of its time one map, $\Phi_1$.

%------------------------------------------------------------------------------------------ 
\begin{Thm}\label{t:OrnsteinWeissGeoFlow} [Ornstein-Weiss  1973, \cite{OrnsteinWeiss}]    The geodesic flow on the unit tangent  bundle of a finite volume hyperbolic surface (or orbifold) is a Bernoulli flow of finite entropy.
\end{Thm}
%------------------------------------------------------------------------------------------  
 
A nice application of the following result can be found in Haas' study of interval maps given by a single M\"obius transformation, \cite{Haas}.  

%------------------------------------------------------------------------------------------ 
\begin{Thm}\label{t:rychlikBernoulli} [Rychlik 1983, \cite{Rychlik}]     Suppose that $T$ is a piecewise monotonic interval map with a unique invariant probability measure that is equivalent to Lebesgue measure; is such that every nonempty open subset is mapped onto the interval by some power of $T$;   and,  whose Jacobian $T'$ is such that $|1/T'|$ has bounded variation.   Then the natural extension of $T$ is a Bernoulli system. 
\end{Thm}
%------------------------------------------------------------------------------------------ 

  Two dynamical systems are called {\em quasi-isomorphic} if their natural extensions are isomorphic.  See \cite{Weiss} for this term and pointers to the literature for examples of non-isomorphic but quasi-isomorphic systems.

\subsection{Terse review of the setting of \cite{CaltaKraaikampSchmidt}}\label{ss:CKS}  We call on this material for applications in \S~\ref{ass:egAlpOne}, and for examples in \S\S~\ref{ss:bddAndFullGivesErgod} and \ref{ss:FiniteQuiltingClose} (including the motivating figures, 
Figure~\ref{f:smallAlpQuilt}  and Figure~\ref{f:quiltLargeAlpLessThanDelta}).   As well, Remark~\ref{rmk:crossBernHoldsForCKSsys} calls on this material to suggest how some of our results can be used in \cite{CaltaKraaikampSchmidtPlanar}.

  As mentioned in the introduction,  we came to the present work in search of tools to further the study of a one-parameter family  of piecewise M\"obius interval maps associated to each of a countably infinite family of triangle Fuchsian groups begun in  \cite{CaltaKraaikampSchmidt}.   Here we give a quick review of some of the notation and terminology from that article.     Besides providing direct motivation for the work here,  as with the Nakada $\alpha$-continued fractions,  this material affords a setting for our illustrative applications throughout this paper.

 %---------------------------------Figure Plot m=n=3, alpha = 0.14 Lebesgue version ----------------------------------------------
\begin{figure}[h]
\scalebox{.33}{
\noindent
\begin{tabular}{rl}
\includegraphics{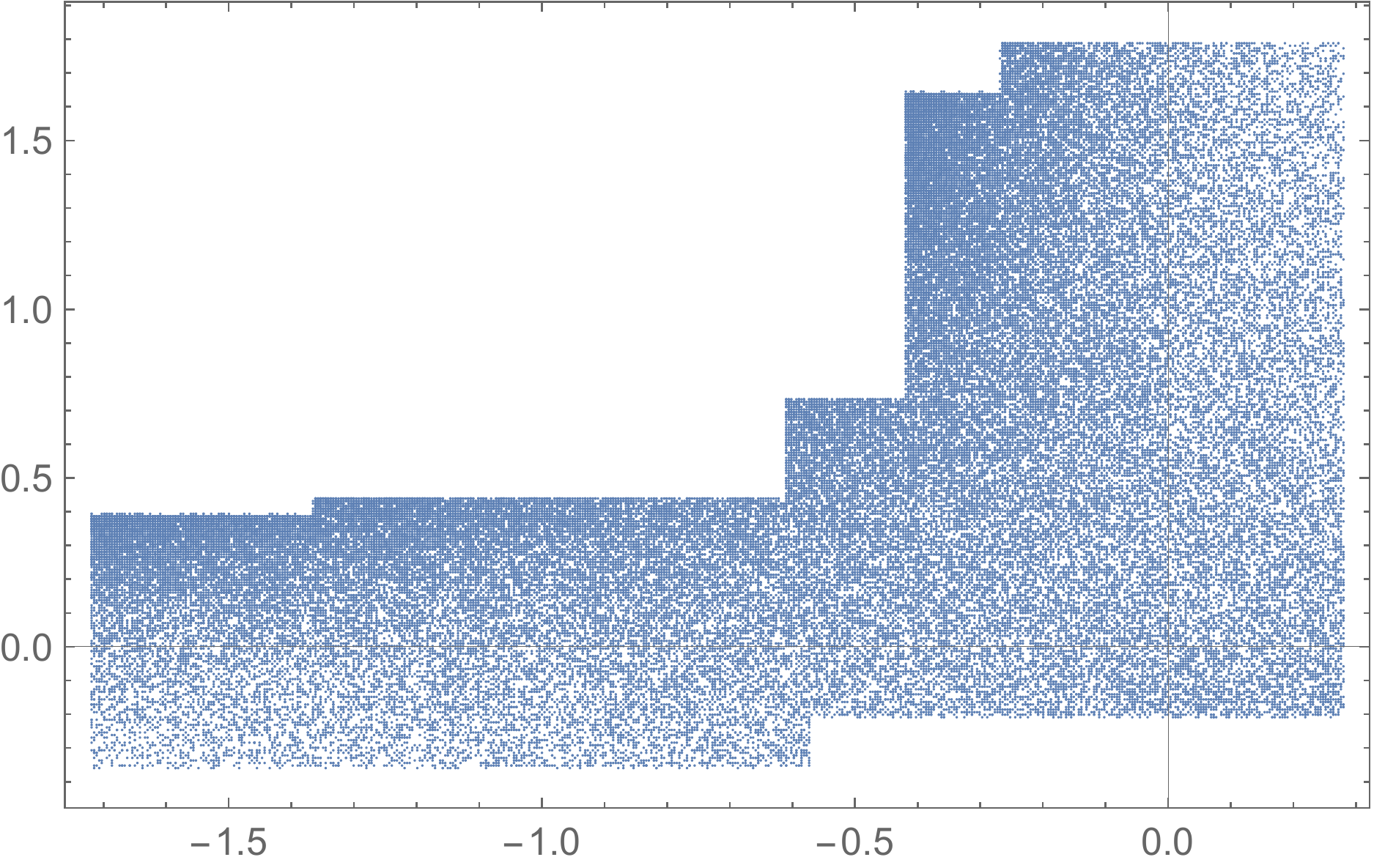}
%Here $2\epsilon_{3,3} = G= (1 +\sqrt{5})/2$ and $2\gamma_{3,3} = g^2 = (G-1)^2$ correspond to the main division points in the interval of the parameter $\alpha$.}
&
\phantom{fourscore and twenty}
\includegraphics{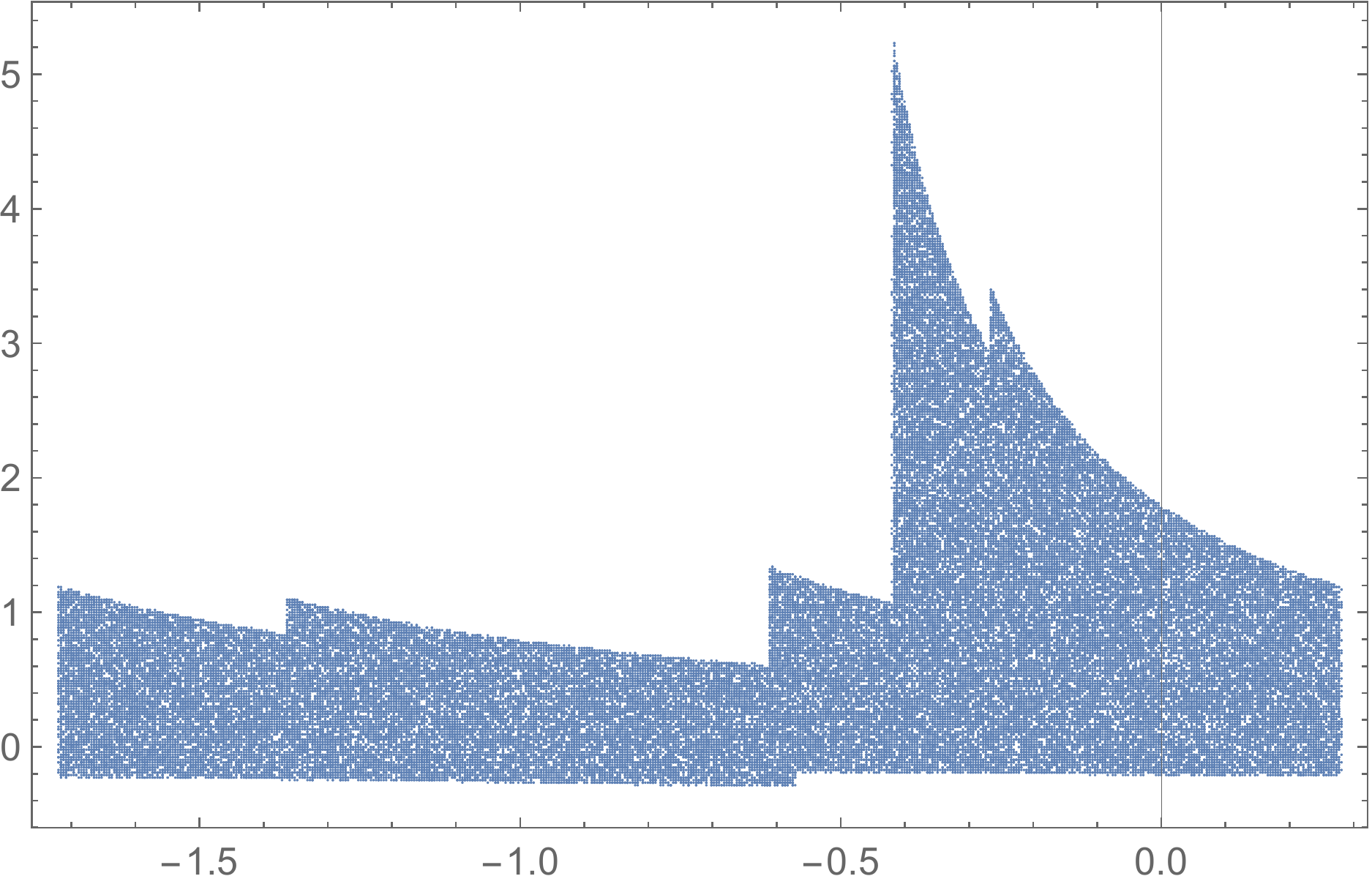}
\end{tabular}
}
\caption{Approximate plots of 100,000 points of a $\mathcal T_{3, 0.14}$-orbit (left), and its image under  $\mathcal Z(x,y)$ of \eqref{eqZmap}.}%
\label{f:standVsLebesgue}
\end{figure}
%------------------------------------------------------------------------------

For integer  $n\ge 3$  we let $\nu = \nu_n =  2 \cos \pi/n$ and 
\[ t  := t_{n} = 1 + 2 \cos \pi/n.\]

We use the group  $G_{n}\subset \text{PSL}(2, \mathbb R)$ generated by 
%----
\begin{equation}\label{e:generators} 
A = \begin{pmatrix} 1& t\\
                                      0&1\end{pmatrix},\,  C = \begin{pmatrix} -1& 1\\
                                      -1&0\end{pmatrix}\,.
\end{equation}
%-----
(We will have no direct use of the matrix $B = A^{-1}C$ of \cite{CaltaKraaikampSchmidt}.    Furthermore,  that paper studies a larger collection of groups, indexed by a pair of integers, $(m,n)$; here we have set the index $m$  equal to 3.)  The group  $G_n$ is a Fuchsian triangle group of signature $(0; 3, n, \infty)$.

 \bigskip
 Fix $n\ge 3$. For each  $\alpha \in [0,1]$, let  $\ell_0(\alpha) = (\alpha - 1)t,    r_0(\alpha) = \alpha t$ and $\mathbb I_{\alpha} = \mathbb I_{n, \alpha} =  [\ell_0(\alpha), r_0(\alpha))$. 
Our interval maps are piecewise M\"obius,  of the form 
%-----
\begin{equation}\label{e:theMaps} 
T_{\alpha} = T_{n,\alpha}:   [\ell_0(\alpha), r_0(\alpha)] \to  \mathbb I_{\alpha}, \;\;\; x \mapsto A^{k} C^{l}\cdot x
\end{equation}
%-----
where $\ell\in\{1,2\}$ is minimal such that $C^{l}\cdot x \notin  \mathbb I_{\alpha}$
 and $k$ is the unique integer such that then $ A^{k} C^{l}\cdot x \in \mathbb I_{\alpha}$.    We call $b_{\alpha}(x)= (k,l)$ the $\alpha$-digit of such an $x$, and say that $x$ lies in the cylinder $\Delta_{\alpha}(k,l)$.   That this does give continued fraction-like expansions of real numbers is shown in \cite{CaltaKraaikampSchmidt}.

 The parameter interval is naturally partitioned, 
 $(0, 1] = (0, \gamma_n)\cup [\gamma_n, \epsilon_n) \cup [\epsilon_n, 1]$,  where $\alpha < \gamma_n$ if  $\forall x \in [\ell_0(\alpha), r_0(\alpha)]$ the $\alpha$-digit $(k,l)$ of $x$ has $l=1$ and $\alpha \ge \epsilon_n$ if and only if both $\alpha > \gamma_n$ and the $\alpha$-digit of $\ell_0(\alpha)$ equals  $(k,1)$ with $k \ge 2$. See (\cite{CaltaKraaikampSchmidt}, Figure~4.1) for plots indicating dynamical behavior related to this partition.   We informally refer to the set of $\alpha< \gamma_n$ as the {\em small $\alpha$},  and all others as  {\em large} $\alpha$; see Figures~\ref{f:standVsLebesgue} and \ref{f:quiltLargeAlpLessThanDelta}.     For small $\alpha$, we use {\em simplified digits:} since $l =1$ we only report the exponent of $A$;  in this setting we use $\underline{d}_{[1, \infty)}^{\alpha}, \overline{d}_{[1, \infty)}^{\alpha}$  in place of $\underline{b}_{[1, \infty)}^{\alpha}, \overline{b}_{[1, \infty)}^{\alpha}$, respectively.    In general, the leftmost cylinder (whose left endpoint is $\ell_0(\alpha)\,$) and the rightmost (with right endpoint $r_0(\alpha)\,$) are possibly non-full; in the case of small $\alpha$,  these are the only possible non-full cylinders. 
 
  For large $\alpha$,  since the $T_{\alpha}$-image of
 \begin{equation}\label{e:bFrak}
 \mathfrak b_{\alpha} = C^{-1}\cdot \ell_0(\alpha)
 \end{equation}
is exactly the image  of $\ell_0(\alpha)$, the cylinder which has  $ \mathfrak b_{\alpha}$ as its left endpoint is also 
a candidate for non-fullness.   Note that for $r_0(\alpha) \ge x \ge \mathfrak b_{\alpha}$ we have $T_{\alpha}(x) = A^k C^2 \cdot x$ for some $k \in \mathbb Z$. %Thus,  $ \mathfrak b_{\alpha}$ is the boundary value between 

From this last, if  $\alpha< \alpha'$ then  $\mathfrak b_{\alpha} < \mathfrak b_{\alpha'}$  and if $\alpha, \alpha'$ are sufficiently close so that the $\alpha$-digit of the left endpoint of $\mathbb I_{\alpha}$ equals $(k, 1)$ as does also the $\alpha'$-digit of the left endpoint of $\mathbb I_{\alpha}$, then there are $x \in \Delta_{\alpha}(k, 2) \cap \Delta_{\alpha'}(1, 1)$, see Figure~\ref{f:quiltLargeAlpLessThanDelta} for an indication of such a situation.  With this condition, $T_{\alpha}(x) = (A^kC)A^{-1}\cdot T_{\alpha'}(x)$,  but the $T_{\alpha}$- and $T_{\alpha'}$-images of any other $x \in \mathbb I_{\alpha}\cap \mathbb I_{\alpha'}$ either agree, or differ by  an application of $A$ or $A^{-1}$.    We consider this matter in Proposition~\ref{p:closeNbhsNotJustShiftChangesQuiltTogether}.

 A main result of \cite{CaltaKraaikampSchmidt} is that there is a notation of {\em matching intervals} (there called synchronization intervals),  and for each $n$ the complement of the union of the matching intervals is a Lebesgue null subset of the parameter interval $[0,1]$.  We again call this complement the {\em exceptional set}, $\mathcal E = \mathcal E_n$.

We can define a map $\mathcal T_{n, \alpha}$ on that portion of the plane fibering over the interval of definition of  $T_{n, \alpha}$, and hope for finding planar extensions.   Again, see  Figures~\ref{f:standVsLebesgue} and \ref{f:quiltLargeAlpLessThanDelta}.

\section{Bounded non-full range  and finiteness of $\Omega$ implies ergodicity}\label{s:BddNonFullRange}  In this section, we state and prove Theorem~\ref{t:evenExpanErgoNaturally} and also give an application of it.

We use the term {\em eventually expansive} to describe an interval map $T$ having some compositional power $r$ that is expansive,  thus there is some $c>1$ so that all $x$ in the domain of $T^r$ satisfy  $|(T^r)'(x)|\ge c$.   

 As stated in the Introduction, we studied an infinite countable collection of one parameter family of piecewise M\"obius interval maps  in  \cite{CaltaKraaikampSchmidt}.     
  As opposed to say the Nakada $\alpha$-continued fractions,   almost all of the maps we considered there are {\em not expansive} maps.  A direct proof that each is eventually expansive seems tedious at best; this motivated us to seek   a general result that can be easily applied to deduce eventual expansitivity.   We give such a result here.    

One expects sufficiently nice continued fraction maps to be ergodic with respect to some measure which is absolutely continuous with respect to Lebesgue measure; the easiest setting to prove such results is when a Markov condition is fulfilled.   In the setting of \cite{CaltaKraaikampSchmidt}, and in many cases of continued fraction like maps,  Markov properties do not hold.   In Definition~\ref{d:Def} below,  we introduce a property that is often fulfilled  in these settings.  This property and basic finiteness conditions satisfied by a planar extension for a map then imply ergodicity and more.

%------------------------------------------------------------------------

\subsection{Adler's `Folklore Theorem'}  

Making an initial approach of R\'enyi much more practical, Adler \cite{AdlerRevisit, AdlerCF} gave conditions  implying that an interval map $f$ has a unique ergodic measure that is equivalent to Lebesgue measure.    In his afterword to \cite{BowenInvMeas},  Adler sketched how to loosen one of his original conditions, with the following  result (to which he referred there as a {\em folklore theorem}). 

%------------------------------------------------------------------------ 
 \begin{Thm}\label{t:adler}[Adler, 1979;   \cite{AdlerRevisit, AdlerCF, BowenInvMeas}] Suppose that $f$ is an interval map such that:
 
 \begin{enumerate} 
 \item[i.)]  All cylinders of $f$ are full;
 \item[ii.)]   $f$ is twice differentiable; 
 \item[iii.)]    $f$ is eventually expansive;
 \item[iv.)]    there is a finite bound on $|f''(x)|/ f'(x)^2$ for $x$ in the domain of $f$.  
 \end{enumerate}
 
 Then $f$ has a unique ergodic probability measure that is equivalent to Lebesgue measure.  
 \end{Thm} 
%------------------------------------------------------------------------  

%------------------------------------------------------------------------ 
\begin{Rmk}\label{rmk:adlerFolkAndEquiv}  Note that there can be at most one ergodic probability measure for $f$ that is equivalent to Lebesgue measure.  The reason for this is that any two ergodic measures are mutually singular; see say \cite{einWar}.   Thus, the existence result in the theorem implies the uniqueness result.  
\end{Rmk}
%------------------------------------------------------------------------ 

\subsection{Statement of first main result}\label{ss:bddAndFullGivesErgod}
  Our first main result shows that, in short, boundedness of fibers of $\Omega$ and a full cylinder for the given interval map implies ergodicity and more.

\subsubsection{Cylinder covering property}   We introduce a condition  inspired by a condition introduced by Ito-Yuri \cite{ItoYuri}.   Their {\em finite range property} holds for a map $f$   if there is a finite set of measurable subsets, say $\mathcal R = \{V_0, \dots, V_N\}$ of the interval such that for every $n \in \mathbb N$ the image under $f^n$ of any rank $n$ cylinder is in $\mathcal R$.  Of course, if every cylinder is full, then that $f$ has the property as is shown by letting $\mathcal R$ be the singleton consisting of the interval of definition itself.    That is, the finite range property  is a weakening of Adler's condition of having full cylinders.       We introduce a property implied by the finite range property  whenever there are infinitely many full cylinders for $f$.
\medskip

%------------------------------------------------------------------------ 
\begin{Def}\label{d:Def} We say that an interval map has {\em bounded non-full range} if there is a full cylinder such that the orbits of the endpoints of all  non-full  cylinders  (and hence of all cylinders) avoid  the interior of this full cylinder.    
\end{Def} 
%------------------------------------------------------------------------ 

  Thus,  under this property,   the range under all positive compositional powers of the interval map of each endpoint of any non-full cylinder   is  bounded away from  the interior of some full cylinder. 
\medskip 

 To illustrate the ease of verification of our property,  we show that the bounded non-full range property holds for a large subset of one of the most studied families of continued fraction maps,   the Nakada $\alpha$-continued fractions  \cite{N}.    Recall the  review of these  in \S~\ref{ssHitoshi}.

%------------------------------------------------------------------------ 
 \begin{Lem}\label{p:NakadaBddRange}  Let $\mathcal E$ denote the exceptional set for Nakada's $\alpha$-continued fractions.  Both every rational $\alpha \in (0,1]$ and every $\alpha \in \mathcal E$  is such that Nakada's $\alpha$-continued fraction map $T_{\alpha}$ has bounded non-full range.            
 \end{Lem}
%------------------------------------------------------------------------  
\begin{proof}  Whenever $\alpha \in \mathbb Q \cap (0,1]$,  the endpoints of $[\alpha-1, \alpha)$ have finite $\alpha$-expansion, both eventually reaching zero.   Since for any $\alpha$, of the infinitely many cylinders of $T_{\alpha}$, the only possible non-full cylinders are those of these endpoints. Furthermore,  whether each is non-full or not depends only on the image of $\alpha-1$ or  $\alpha$, respectively.   For these rational values of $\alpha$, we thus have that the corresponding $T_{\alpha}$  have bounded non-full range.    

Recall that $T_1$ is the regular continued fraction map.    In the proof of (\cite{KraaikampSchmidtSteiner}, Lemma~6.8;  note that matching intervals are called synchronization intervals in that paper)  a result of   \cite{CT} is verified: $\alpha \in \mathcal E$ if and only if $T_{1}^n(\alpha) \ge  \alpha$ for all $n \in \mathbb N$.  In particular, if $\alpha \in \mathcal E$ then its regular continued fraction expansion is of the form $[0; a_1, a_2, \dots]$ with $a_1 \ge a_n$ for all $n>1$ and in particular the $a_i$ take on only finitely many values.   Now,  given this expansion of $\alpha \in \mathcal E$,  from (\cite{KraaikampSchmidtSteiner}, Proposition~4.1) the $\alpha$-expansion of $\alpha-1$ has digits contained in a finite set and hence from   (\cite{KraaikampSchmidtSteiner} Lemma~6.7) also the $\alpha$-digits  of $\alpha$ itself are contained in a finite set.     Since the $\alpha$-digits correspond to cylinders,   $T_{\alpha}$ is of bounded non-full range.    
\end{proof} 
%------------------------------------------------------------------------ 

 The above illustrates a setting where the finite range property fails, but our property holds.  This is due to the fact that although for any $\alpha\in \mathcal E$ we have that both endpoints of $\mathbb I_{\alpha}$ have bounded digits, their orbits include infinitely many distinct points.
\medskip

 We further illustrate the verification of the property.
%------------------------------------------------------------------------ 
\begin{Eg}\label{eg:BddRangeCKS}   As recalled in \S~\ref{ss:CKS},  in \cite{CaltaKraaikampSchmidt} we study families of maps $T_{n,\alpha}, n>3, \alpha \in [0,1]$ related to certain Fuchsian triangle groups, $G_n$.  Fixing $n$, and   $\alpha$,  the corresponding maps are of the basic form $x \mapsto A^kC^l\cdot x$, for various integers $k,l$, where $A, C \in G_n$ are explicitly given in \S~\ref{ss:CKS}.  For each $n$, there is a real value $\gamma_n$ such that for all $\alpha<\gamma_n$, the map indexed by $\alpha$ is of the simpler form $x \mapsto A^kC\cdot x$.

 Fix $n$ and   a `small' $\alpha$, thus  $\alpha \in (0,\gamma_{n})$. In this setting, --- similar to the case of the Nakada $\alpha$-fractions ---  there are only at most two non-full cylinders, the leftmost cylinder whose left endpoint is $\ell_0(\alpha)$ and the rightmost with right endpoint $r_0(\alpha)$.  
 Suppose further that $\alpha$ is in the exceptional set $\mathcal E_n$. 
 Subsection~4.5 of \cite{CaltaKraaikampSchmidt} shows that  the $T_{n,\alpha}$-orbit of $\ell_0(\alpha)$ meets only the two leftmost cylinders, while the $T_{n,\alpha}$-orbit of $r_0(\alpha)$ meets only the two  rightmost cylinders.         From this, we find that each of these maps has bounded non-full range.  Indeed, here there are infinitely many full cylinders avoided by the orbits in question.   
\end{Eg}
%------------------------------------------------------------------------ 

\subsubsection{Statement of result} 
%------------------------------------------------------------------------ 
 \begin{Thm}\label{t:evenExpanErgoNaturally} Suppose that  $T$ is a piecewise M\"obius map on an interval $I$ of finite Lebesgue measure and $\mathcal T: \Omega \to \Omega$ is a planar extension for $T$  
 such that 
 \begin{enumerate}
\item[a)]   the vertical fibers of $\Omega$  are of positive Lebesgue measure bounded away from both zero and infinity;
\item[b)]  the vertical fibers are bounded away from the locus of $y = -1/x$;
\item[c)]  $T$ has at least one full cylinder for which the set of ratios of 
the Lebesgue measure  of the $\mathcal T$-image of each vertical fiber above this cylinder to the Lebesgue measure of its receiving fiber is bounded away from zero and one;
\item[d)]  $T$ has bounded non-full range.
\end{enumerate}

 Let $\mu'$  be the normalization of $\mu$ to a probability measure on $\Omega$ and $\nu$ be the marginal measure  of $\mu'$.  Also let $\mathscr B, \mathscr B'$ denote the Borel algebras of $I, \Omega$ respectively. Then

 \begin{enumerate}
 \item[i.)]   $0< \mu(\Omega) < \infty$;
 \item[ii.)]     $T$ is eventually expansive;
 \item[iii.)]     $T$ is ergodic with respect to $\nu$;
 \item[iv.)]       the system $(\mathcal T, \Omega, \mathscr B', \mu')$ is the natural extension of  $(T, I, \mathscr B, \nu)$.  In particular,  the two dimensional system is also ergodic.  
\end{enumerate}
\end{Thm}
%------------------------------------------------------------------------ 

That $\mu(\Omega)$ is finite is easily seen.  We prove the remaining conclusions in three steps. 

\bigskip 

 \subsection{Eventual expansivity}
%------------------------------------------------------------------------ 
 \begin{Prop}\label{p:evenExpan}  Under the hypotheses (a)--(c) of Theorem~\ref{t:evenExpanErgoNaturally}, 
 $T$ is eventually expansive.
\end{Prop}
%------------------------------------------------------------------------ 

 \begin{proof}  Let $\widehat{\mathcal T}: \Sigma \to \Sigma$ denote the conjugate two-dimensional system where $\Sigma$ is the image of $\Omega$ under  the map $\mathcal Z(x,y)$ of \eqref{eqZmap}.   Lebesgue measure is invariant for this conjugate system.   In particular,  for each $x \in I$ the vertical fiber  $F_x \subset \Sigma$ projecting to $x$ is mapped by $\widehat{\mathcal T}$ into the vertical fiber   $F_{Tx}$,  with derivative along $F_x$ constantly equal to $(T'(x))^{-1}$.    Hypotheses  (a) and (b) imply that there are positive finite bounds $0 < b<B$  on the one-dimensional Lebesgue measure of the $F_x$.       
   
The third hypothesis also carries over to the conjugate system.   Let us temporarily use vertical bars to indicate the one-dimensional Lebesgue measure on vertical fibers of $\Sigma$.   Denoting the chosen full cylinder by $\mathcal C$, we have that the set of ratios
\[ \Big\{ \dfrac{|\widehat{\mathcal T}(F_x) |}{|F_{Tx}|}\;:\; x \in \mathcal C\Big\}\]
is also bounded away from zero and one.    We can thus find a $\rho$ with  $0< \rho <1$ such that $1-\rho$ is  a lower bound and $\rho$ an upper bound.    

Now, for any $z \in I$,  the fiber at $z$ is the union of the images of the fibers over the preimages of $z$; that is, 
$F_z = \cup_{Tx = z}\, \widehat{\mathcal T}(F_x)$.   Given $x\in I\setminus \mathcal C$, set $z = Tx$.   Since  $\mathcal C$ is a full cylinder,  there is some $x'\in \mathcal C$ such that $T(x') = z$  and hence $\mid\widehat{\mathcal T}(F_{x'})\mid$ gives at least $1- \rho$ of $|F_{Tx}|$.     It follows that  $|\widehat{\mathcal T}(F_x) |   \le \, \rho |F_{Tx}|$.   Hence,  for all $x \in I$, we have   $|\widehat{\mathcal T}(F_x) |   \le \, \rho |F_{Tx}|$. 

 Recall that $\widehat{\mathcal T}$ has constant derivative along each vertical fiber.     
  Thus, ratios of measures are preserved; in particular
\[  \dfrac{ \mid\widehat{\mathcal T}^2(F_x) \mid}{\mid \widehat{\mathcal T}( F_{Tx} )\mid} =  \dfrac{ \mid\widehat{\mathcal T}(F_x ) \mid}{\mid F_{T x} \mid}.\]
Using a telescoping expansion and substituting the above, we deduce  
  
\[  \dfrac{ \mid\widehat{\mathcal T}^2(F_x) \mid}{\mid F_{T^2 x} \mid} =  \dfrac{ \mid\widehat{\mathcal T}(F_x ) \mid}{\mid F_{T x} \mid}\cdot \dfrac{ \mid\widehat{\mathcal T}(F_{Tx}) \mid}{\mid F_{T^2 x}\mid}  \le  \, \rho^2
\]
and similarly for higher powers.   
Now let $r \in \mathbb N$ be such that $\rho^{r-1} B < b$.     Then  for any $x \in I$ we have   $b \le |F_x|$ but  $|\widehat{\mathcal T}^r(F_x) | \le \rho^r  \mid F_{T^r x} \mid \le \rho^r B < \rho b$.    Thus,  $|\widehat{\mathcal T}^r(F_x) | <\rho  |F_x|$.   Since  $\widehat{\mathcal T}^r$ also preserves two-dimensional Lebesgue measure, we must have that  $\vert (T^r)'(x)\vert  >\rho^{-1}$, with now the vertical bars denoting the absolute value. Therefore,  $T^r$ is expansive.  
\end{proof} 

\subsection{Ergodicity}

%------------------------------------------------------------------------ 
 \begin{Prop}\label{p:ergodic} Under the hypotheses of Theorem~\ref{t:evenExpanErgoNaturally},
$T$ is ergodic with respect to   $\nu$.    
\end{Prop}
%------------------------------------------------------------------------ 
 \begin{proof}  Since $T$ has bounded non-full range, there is a largest interval, say  $J$, comprised of full cylinders avoided by the orbits of  the endpoints of all non-full cylinders.   Let 
$\widetilde T$ be  the first return map to $J$ of orbits of $T$. 
We will show that the conditions for the Adler result, Theorem~\ref{t:adler}, hold for $f = \widetilde T$.    Since $\nu$ is equivalent to Lebesgue, so is the probability measure it induces on $J$.     As per  Remark~\ref{rmk:adlerFolkAndEquiv}, the Adler result will then imply that this induced measure itself is ergodic. Ergodicity of a map induced from a general $f$ implies that  $f$ itself is also ergodic under reasonable hypotheses  (see Theorem~17.2.4 of \cite{Schw}).     Thus,  the ergodicity of $\nu$ will follow. Again by the remark, this implies that $\nu$ is the unique ergodic invariant measure for $T$ that is equivalent to Lebesgue measure.
 
 The cylinders of $\widetilde T$  are of the form $Q_{\beta}$ where $\beta = (a_1, \dots, a_m)$ with $a_1$ an index of a $T$-cylinder  meeting $J$, $[a_1, \dots, a_m]$ a rank $m$ cylinder for $T$, and $Q_{\beta} = [a_1, \dots, a_m] \cap T^{-m}(J)$.   By Poincar\'e recurrence,  up to measure zero $J$ is the union of
the $Q_{\beta}$.  Since $J$ consists of full cylinders for $T$ which the $T$-orbits of the endpoints of  $T$-cylinders never enter, each $Q_{\beta}$ is a full $\widetilde T$-cylinder.

    The corresponding planar map $\widetilde{\mathcal T}$ is bijective up to $\mu$-measure zero on $\widetilde{\Omega}$,  the region defined by deleting the portion of $\Omega$ projecting to the complement of $J$.     In particular,  the arguments in the proof of Theorem~\ref{t:evenExpanErgoNaturally} apply, and  thus $\widetilde T$ is eventually expansive.

\smallskip % to move tilde away from previous line
 We have ensured that $\widetilde T$ has full cylinders, and is eventually expansive.  Our construction also preserves the property of being twice differentiable.   The crux of the matter is thus to show that Adler's fourth condition holds.   Since $T$ is a piecewise M\"obius map,  certainly for $\nu$-a.e.  $x$,  there is some matrix $M = \begin{pmatrix}a&b\\c&d\end{pmatrix}$ such that $\widetilde T(x) = M\cdot x$.   The first derivative here is $\det M\, (c x + d)^{-2}$,  we must bound  $|c (c x + d)|$ over all such $x, M$.    
 
   Since the vertical fibers of $\mathcal Z(\Omega)$ are bounded, so are those of $\mathcal Z(\widetilde{\Omega})$.  That is, every  vertical fiber has Lebesgue measure in an interval $[b,B]$ bounded away from zero and infinity.      To avoid notational unpleasantries, let us  now use $\widehat{\mathcal T}$ to denote the conjugate of $\widetilde{\mathcal T}$ acting on $\mathcal Z(\widetilde{\Omega})$; see \eqref{eqLebMap}.     Recall that restricting   $\widehat{\mathcal T}$ to $F_x$ (the vertical fiber at $x$) defines a map whose derivative equals  $\vert \widetilde T'(x)\vert^{-1}$.  It follows that $\vert \widetilde T'(x)\vert^{-1} \le B/b$ for all $x \in I$.  Hence, the set of  values $(cx + d)^2 y$ is bounded.  The boundedness of the fibers directly implies that the values $(cx + d)^2 y - c (cx + d)$ are bounded.   We conclude that  $c (cx+d)$ is bounded throughout $I$.  This implies that Adler's  fourth condition holds.    Therefore,  $\widetilde T$ is ergodic and, as argued in the first paragraph  of this proof, the  ergodicity of $\nu$ holds.
\end{proof}

%------------------------------------------------------------------------ 
\begin{Rmk}\label{rmk:inducing}  Given the above argument, one could ask whether it is always possible to induce past non-full cylinders and be sure that Adler's conditions hold.    We strongly doubt this, as in general the return iteration number to the complement of those cylinders will be unbounded.   As Zweim\"uller \cite{Z}  states,  this in general  will cause Adler's condition (4) to fail.  In our setting, of course, such ``explosion" is impossible due to the boundedness of the vertical fibers of $\mathcal Z(\Omega)$.
\end{Rmk}
%------------------------------------------------------------------------ 
 
\subsection{Natural extension}\label{ass:natExt}   

Arnoux in particular has been a proponent of solving for planar presentations using properties of the interval map.    In particular,  the main results of \cite{ArnouxSchmidtNatExt} imply that  (1) if  a piecewise M\"obius interval map $T:I \to I$ is (eventually) expansive then its associated naive two-dimensional map $\mathcal T:  I\times \mathbb R \to I\times \mathbb R$ induces a contraction on the complete metric space of compact subsets of this product, where a modified Hausdorff metric is used.   (The contraction sends a compact $K$ to the closure of the union of the $\mathcal T_{\beta}(K_{\beta})$,   where $K_{\beta}$ is the portion of $K$ projecting to the $\beta$-cylinder.)  And, (2)   when the fixed point, say $\Omega$, of this contraction has positive measure and $\mathcal T$ is bijective on $\Omega$ up to measure zero, then this two-dimensional system is a planar natural extension of $T$.     

 We thus find the following. 
%------------------------------------------------------------------------ 
 \begin{Prop}\label{p:natExtFromASnatExt} Under the hypotheses of Theorem~\ref{t:evenExpanErgoNaturally}, 
 $(\mathcal T, \Omega, \mathscr B', \mu')$ is the natural extension of  $(T, I, \mathscr B, \nu)$. 
\end{Prop}
%------------------------------------------------------------------------  

\subsection{Application to each of an infinite collection of maps}\label{ass:egAlpOne}   The goal of this subsection is to apply Theorem~\ref{t:evenExpanErgoNaturally} to  a family of maps whose possible ergodicity was unknown, and indeed, whose planar natural extensions had not   been determined.  We reach this goal in Corollary~\ref{c:accErg}.  

 Our maps are related to those studied in \cite{CaltaKraaikampSchmidt}, again refer to 
\S~\ref{ss:CKS}. The dynamics of the maps $T_{n,1}$ (among others) are presented in Section 3 of  \cite{CaltaKraaikampSchmidt}.     Here we give a planar extension for each of these maps.  Each is of infinite mass; by ``accelerating" each of the interval maps past the domain of the parabolic element of the underlying group whose fixed point is responsible for the infinitude of the mass, we  obtain an interval map of invariant probability measure, as verified by applying Theorem~\ref{t:evenExpanErgoNaturally}.

%--------------------------------- Figure Omega_{ n,1}  vertical ----------------------------------------------
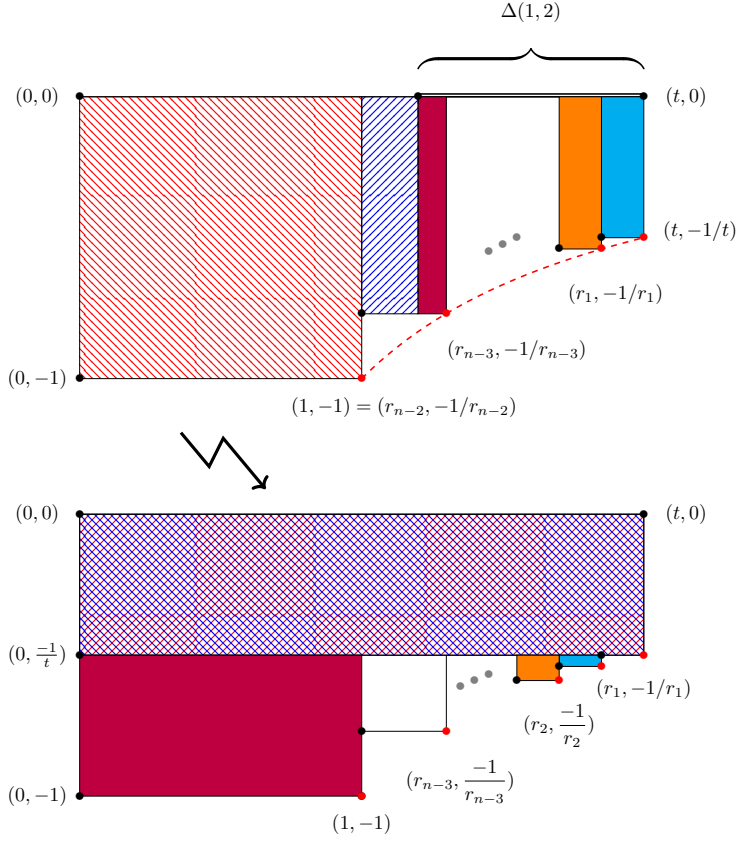
\begin{figure}[h]
\scalebox{.75}{
\noindent
\begin{tabular}{l}
%$$
\begin{tikzpicture}[x=5cm,y=5cm] 
%% Basic outline
\draw  (0.0, 0.0)--(2,0);
\draw[thick]  (1.2, 0.0)--(2,0);
\draw[thick]  (1.2, 0.01)--(2,0.01);
\draw[pattern=north west lines, pattern color=red]  (0,0)--(0,-1)--(1,-1)--(1,0) -- cycle;
\draw[pattern=north east lines, pattern color=blue]  (1,0)--(1,-0.77)--(1.2,-0.77)--(1.2,0) -- cycle;
\draw[fill=purple]  (1.2,0)--(1.2,-0.77)--(1.3,-0.77)--(1.3,0)--cycle;
\draw[fill=orange]  (1.85,0)--(1.85,-0.54)--(1.7,-0.54)--(1.7,0)--cycle;
\draw[fill=cyan]  (2,0)--(2,-.5)--(1.85,-.5)--(1.85,0)--cycle; 
\draw [dashed]  (1.2,-0.7)--(1.2,0); 
%Brace and label
\draw [decorate, ultra thick,
    decoration = {calligraphic brace,mirror,amplitude=10pt}] (2,0.1) --  (1.2,0.1);
\node at (1.6, 0.3) {$\Delta(1,2)$};     
\node at (1.2,0) {$\bullet$};  
%Kludge for rising dots 
 \foreach \x/\y in {1.45/-0.55,1.5/-0.525, 1.55/-0.5%
} { \node at (\x,\y) [gray]  {$\bullet$}; } 
%Mark vertices (off of infinity)
 \foreach \x/\y in {0/0,0/-1,  1/-0.77,   1.7/-0.54,  1.85/-0.5,   2/0%
} { \node at (\x,\y) {$\bullet$}; } 
%Mark vertices at infinity
 \foreach \x/\y in {1/-1, 1.3/-0.77, 1.85/-0.54,  2/-0.5 %
} { \node at (\x,\y) [red] {$\bullet$}; } 
%Label vertices 
\node at (-0.15,  0) {$(0,0)$};   
\node at (-0.15,  -1) {$(0,-1)$}; 
\node at (1.15,  -1.10) {$(1,-1) = (r_{n-2},-1/r_{n-2})$}; 
\node at (2.15,  0) {$(t,0)$}; 
\node at (2.2,  -0.475) {$(t,-1/t)$}; 
\node at (1.9,  -0.7) {$(r_1,-1/r_1)$}; 
\node at (1.55,  -0.9) {$(r_{n-3},-1/r_{n-3})$};
% Locus of infinity
 \draw[domain= 1:2, smooth, variable=\x, red, thick, dashed] plot ({\x}, {-1/\x});
\end{tikzpicture}
\\
\begin{tikzpicture}[x=5cm,y=1cm] 
\node at (2, 0) {\phantom{here}};
%\draw[->, ultra thick] (2.55, .2) -- (2.65, -0.3) -- (2.7, 0.15) -- (2.85,-0.6);
\draw[->, ultra thick] (2.55, .2) -- (2.65, -0.4) -- (2.7, 0.1) -- (2.85,-0.8);
\end{tikzpicture}
\\
\begin{tikzpicture}[x=5cm,y=5cm] 
%% Basic outline
\draw  (0.0, 0.0)--(2,0);
\draw[pattern=north east lines, pattern color=red]    (0,0)--(2,0)--(2,-0.5)--(0,-0.5) -- cycle;
\draw[pattern=north west lines, pattern color=blue]  (0,0)--(2,0)--(2,-0.5)--(0,-0.5) -- cycle;
\draw  ( (1,-0.77)--(1.3,-0.77)--(1.3,-0.5);
\draw[fill=cyan]   (1.85,-0.5)--(1.85,-0.54)--(1.7,-0.54)--(1.7,-0.5)--cycle;
\draw[fill=orange]  (1.7,-0.5)--(1.7,-0.59)--(1.55,-0.59)--(1.55,-0.5)--cycle;
\draw[fill=purple]  (0,-0.5)--(0,-1)--(1,-1)--(1,-0.5)--cycle;
%Kludge for rising dots 
 \foreach \x/\y in {1.35/-0.61, 1.4/-0.59, 1.45/-0.57%
} { \node at (\x,\y) [gray]  {$\bullet$}; } 
%
%Mark vertices
 \foreach \x/\y in {0/0,0/-0.5, 0/-1, 1/-1, 1/-0.77,  1.55/-0.59, 1.7/-0.54,1.85/-0.5,  2/0%
} { \node at (\x,\y) {$\bullet$}; } 
%Mark vertices at infinity
 \foreach \x/\y in {1/-1, 1.3/-0.77, 1.7/-0.59, 1.85/-0.54,  2/-0.5 %
} { \node at (\x,\y) [red] {$\bullet$}; } 
%Label vertices 
\node at (-0.15,  0) {$(0,0)$};   
\node at (-0.15,  -0.5) {$(0,\frac{-1}{t})$}; 
\node at (-0.15,  -1) {$(0,-1)$}; 
\node at (1,  -1.10) {$(1,-1)$}; 
\node at (2.15,  0) {$(t,0)$}; 
%\node at (2.2,  -0.475) {$(t,-1/t)$}; 
\node at (2,  -0.62) {$(r_1,-1/r_1)$}; 
\node at (1.7,  -0.75) {$(r_2,\dfrac{-1}{r_2})$}; 
\node at (1.35,  -0.95) {$(r_{n-3},\dfrac{-1}{r_{n-3}})$};
% Locus of infinity
 %\draw[domain= 1:2, smooth, variable=\x, red, thick, dashed] plot ({\x}, {-1/\x});
 %
\end{tikzpicture}
% $$
%
\end{tabular}
}
\caption{Schematic representation of the domain $\Omega_{n, 1}$, and its image under $\mathcal T_{n, 1}, n\ge 3$ as discussed in Proposition~\ref{p:alpOneTwoDdom}.    Compare with (\cite{CaltaKraaikampSchmidt},  left side of Fig.~6).  Here, on the left side, the red dotted curve plots $y=-1/x$ the only points of $\Omega$ on this are the (red) vertices coming from the orbit of $(t, -1/t)$.  (Recall that our measure is given by $d \mu = (1+xy)^{-2}\, dx\, dy\,$ !)  Blocks fibering over intervals whose endpoints are consecutive members of the orbit of $r_0 = t$ under the interval map are filled with solid colors. Red hatching indicates blocks fibering over cylinders indexed by $(i,1), i \in \mathbb N$.  Blue hatching  indicates blocks fibering over cylinders indexed by $(j,2), j\ge 2$.   Images on the right hand side correspondingly colored, except that the cross-hatching indicates lamination from the hatched portions.}%
\label{f:omegaOne}%
\end{figure}
%-------------------------------------------------------------------------------------------

  Fix $n \ge 3$, let $T = T_{n,1}$ and $U = AC (AC^2)^{n-2}$, and note that $t = t_n = r_0(1)$.  From   \cite{CaltaKraaikampSchmidt}, one has $T^i(t) = (AC^2)^{n-2}\cdot t$ for $1 \le i \le n-2$ and   $T^{n-1}(t)= U\cdot t = t$.  Furthermore,  all of the cylinders of $T$ are full except for the right most cylinder,  $\Delta(1,2)$, of endpoints  $\mu +1/t = 1+1/t $ and $t$.   It is easily verified that $(AC^2)^{n-2}\cdot t = 1$. \\

Compare the following with Figure~\ref{f:omegaOne}.
%------------------------------------------------------------------------ 
 \begin{Prop}\label{p:alpOneTwoDdom}     Fix $n\ge 3$ and let $T = T_{n,1}$ and $\mathcal T$ be the usual associated two-dimensional map. 
Let   $r_0, r_1, \dots, r_{n-2}$ be the $T$-orbit of $r_0 = t$.  Then $\mathcal T$ is bijective up to $\mu$-measure zero on    
 \[\Omega = [0,1]\times [-1,0] \;  \cup \;  \bigcup_{i=1}^{n-2}\; [r_i, r_{i-1}]\times [-1/r_{i-1},0].\]
 \end{Prop}
%-----------------------------------------------------------------------
\begin{proof} We have that $(0,1]$  is the union of the cylinders  $\Delta(i,1)$ with $i \in \mathbb N$.  Similarly, $(1, 1 + 1/t]$ is the union of the $\Delta(j,2)$ with $j \ge 2$.    
Since $1+1/t$ lies between $1= r_{n-2}$ and  $r_{n-3}$,   the $y$-fiber of $\Omega$ above each of these  $\Delta(j,2)$ is  $[-1/r_{n-3}, 0]$, whereas every $\Delta(i,1)$ has $y$-fibers given by $[-1,0]$.   Recall that $R$ is given in \eqref{e:justR}; since $RCR^{-1} \cdot 0 = -1$, it follows that $R A^k C R^{-1} \cdot -1 = R A^{k} C^2R^{-1} \cdot 0$ for any $k$.   Hence, each rectangle
$\Delta(k, 1)\times [-1,0]$ is mapped  above
 the image of $\Delta(k, 2)\times [-1/r_{n-3},0]$   so as to share exactly a common horizontal line.

 Now, $AC^2\cdot r_{n-3} = 1$ can be used to show that $RC^2 A C^2 R^{-1} \cdot -1/r_{n-3} = 0$ and a similar observation implies that 
each  rectangle $\Delta(k, 1 )\times  [-1,0]$ is mapped  below the image of 
 $\Delta(k+1, 2)\times [-1/r_{n-3},0]$   so as to share exactly a common horizontal line.

Therefore,  $\mathcal T$ sends $\Omega \cap \{x \le 1+ 1/t\}$ bijectively up to measure zero to  $(0,t]\times [-1/t, 0)$.   Furthermore,  since every $\mathcal T_M$ preserves the locus $y= -1/x$, 
we easily find that  $i = 1, \dots, n-4$,  
\[  \mathcal T_{AC^2}  (\,  [r_i, r_{i-1}]\times [-1/r_{i-1},0]  \,) =    [r_{i+1}, r_i]\times [-1/r_i,  -1/t] .\]
(Of course, when $n\le 4$ we must make appropriate adjustments.)   Furthermore,   $ \mathcal T_{AC^2}$ sends $(1+1/t, r_{n-3}]\times [-1/r_{n-3}, 0]$ to $(0,1]\times [-1, -1/t]$.     The result thus holds.  
\end{proof}

\bigskip 
%--------------------------------- Figure Gamma_{n,1}, the accelerated version ----------------------------------------------
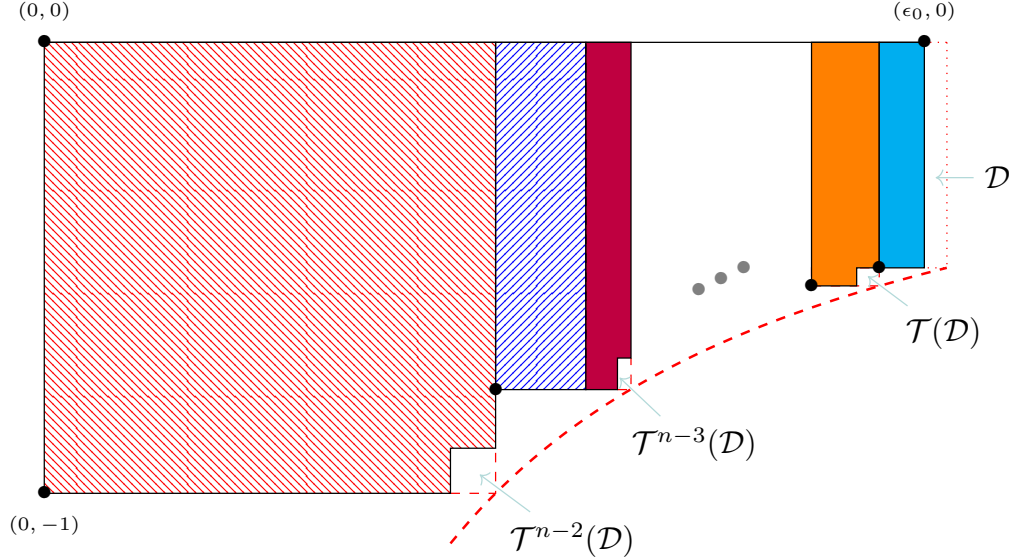
\begin{figure}[h]
\scalebox{1.2}{
\begin{tikzpicture}[x=5cm,y=5cm] 
%% Basic line
\draw  (0.0, 0.0)--(1.95,0);
%Cut out!
\draw [red, thin, dotted]  (2,0)--(2,-0.5)--(1.95,-0.5)--(1.95,0)--cycle; 
\draw [red, thin, dashed]   (1.85,-0.5)--(1.85,-0.54)--(1.7,-0.54)--(1.7,-0.5)--cycle;
\draw [red, thin, dashed]   (1.27,-0.77)--(1.3,-0.77)--(1.3,-0.7);
\draw [red, thin, dashed]    (0.9,-1)--(1,-1)--(1, -0.9);
\draw[pattern=north west lines, pattern color=red]  (0,0)--(0,-1)--(0.9,-1)--(0.9,-0.9)--(1,-0.9)--(1,0) -- cycle;
\draw[pattern=north east lines, pattern color=blue]  (1,0)--(1,-0.77)--(1.2,-0.77)--(1.2,0) -- cycle;
\draw[fill=purple]   (1.2,0)--(1.2,-0.77)--(1.27,-0.77)--(1.27,-0.7)--(1.3,-0.7)--(1.3,0)--cycle;
\draw[fill=orange]  (1.85,0)--(1.85,-0.5)--(1.8,-0.5)--(1.8,-0.54)--(1.7,-0.54)--(1.7,0)--cycle;
\draw[fill=cyan]  (1.95,0)--(1.95,-.5)--(1.85,-.5)--(1.85,0)--cycle; 
\draw [dashed]  (1.2,-0.7)--(1.2,0); 
%Kludge for rising dots 
 \foreach \x/\y in {1.45/-0.55,1.5/-0.525, 1.55/-0.5%
} { \node at (\x,\y) [gray]  {$\bullet$}; } 
%Mark vertices (off of infinity)
 \foreach \x/\y in {0/0,0/-1,  1/-0.77,   1.7/-0.54,  1.85/-0.5,  1.95/0%
} { \node at (\x,\y) {$\bullet$}; } 
%Label vertices 
\node at (0,  0.07) {\tiny{$(0,0)$}};   
\node at (0,  -1.07) {\tiny $(0,-1)$}; 
\node at (1.95,  0.07) {\tiny $(\epsilon_0,0)$}; 
\node at (1.95,  -0.3)[pin={[pin edge=<-, pin distance=12pt]0:{$\mathcal D$}}] {};
\node at (1.8,  -0.5)[pin={[pin edge=<-, pin distance=12pt]315:{$\mathcal T(\mathcal D)$}}] {};
\node at (1.25, -0.72) [pin={[pin edge=<-, pin distance=12pt]274:{$\;\mathcal T^{n-3}(\mathcal D)$}}] {};
\node at (.94, -.94) [pin={[pin edge=<-, pin distance=12pt]295:{$\;\mathcal T^{n-2}(\mathcal D)$}}] {};
%\node at (1.9,  -0.7) {$(r_1,-1/r_1)$}; 
%\node at (1.55,  -0.9) {$(r_{n-3},-1/r_{n-3})$};
% Locus of infinity
 \draw[domain= .9:2, smooth, variable=\x, red, thick, dashed] plot ({\x}, {-1/\x});
\end{tikzpicture}
}
\caption{Schematic representation of the domain $\Gamma$ for the accelerated two-dimensional map.  The domain is given by deleting from $\Omega$ the rectangle $\mathcal D = (\epsilon_0, t]\times [-t, 0]$ and its images under $\mathcal T, \dots, \mathcal T^{n-2}$.  See Lemma~\ref{l:alpOneTwoAccDom}.}%
\label{f:accOne}%
\end{figure}
%-------------------------------------------------------------------------------------------

\bigskip
Since  $U$ is a conjugate (up to sign) of $A^{-1}$ we see that it is a parabolic matrix and thus $t$ is  parabolic fixed point under $T$.  This in a sense is the cause of the invariant measure $\mu$ being infinite on $\Omega$.   Just as in the treatment of $T_{n,0}$ in \cite{CaltaSchmidt}, we ``accelerate"  our map by inducing past  the cylinder (here of rank $n-1$) related to the parabolic element.    There is an easily determined domain of bijectivity for the corresponding two-dimensional map.   (Indeed the following result is a specific case of a general phenomenon.)   Compare the following with Figure~\ref{f:accOne}.
%------------------------------------------------------------------------ 
 \begin{Lem}\label{l:alpOneTwoAccDom}     Fix $n\ge 3$.   Let $\epsilon_0 =  U^{-1}\cdot 0$  and let $g(x)$ be the first return map of $T_{n,1}$ to   $(0, \epsilon_0)$ and $\mathcal G$ the corresponding two-dimensional map.   Then $\mathcal G$  is bijective up to $\mu$-measure zero on    
 \[\Gamma =\Omega \setminus  \cup_{i=0}^{n-2}\; \mathcal T_{(AC^2)^i}(\mathcal D),\]
 where $\mathcal D =  (\epsilon_0, t]\times [-1/t,0]$. 
 \end{Lem}
%-----------------------------------------------------------------------
\begin{proof} By definition,   $g(x) = T^i(x)$ where $i = i(x)\in \mathbb N$ is minimal such that $T^i(x) < \epsilon_0$.    Since   $(\epsilon_0, t] = \Delta(\, (1,1)^{n-2}(1,2)\,)$, we have that  $g(x) = U^j \circ T(x)$ where $j\ge 0$ is minimal such that the image is outside of the rank $n-1$ cylinder $\Delta(\, (1,2)^{n-2}(1,1)\,)$.   

Now,    $(x,y)\in \mathcal D$ if and only if $T(x) \in \Delta(\, (1,2)^{n-2}(1,1)\,)$  and therefore the $\mathcal T$-orbit of $(x,y)$ includes the initial sequence $\{\mathcal T_{(AC^2)^i}(x,y)\}_{i=0}^{n-2}$.   Furthermore,  due to the bijectivity of $\mathcal T$,  any $(x,y) \in \cup_{i=0}^{n-2}\; \mathcal T_{(AC^2)^i}(\mathcal D)$  must belong to a length $n-1$ orbit sequence with an initial orbit element in $\mathcal D$.     

 Suppose now that $(x,y) \in \Gamma$ and $\mathcal T(x,y) \notin \mathcal D$.  From the previous paragraph, we have that   $g(x) = T(x) = M\cdot x$ for some matrix $M$ and hence both $\mathcal G(x,y) = \mathcal T(x,y) = \mathcal T_M(x,y)$ and $\mathcal G(x,y)$ must indeed belong to $\Gamma$.   
 
 On the other hand, if  $(x,y) \in \Gamma$ with $\mathcal T(x,y) \in \mathcal D$, then $T(x) \in (\epsilon_0, t)$ and there is a $j \in \mathbb N$ such that $\mathcal G(x,y) = \mathcal T_{U^j}\circ \mathcal T (x,y)  =  \mathcal T_{AC}\circ \mathcal T_{(AC^2)^{n-2}} \circ \mathcal T_{U^{j-1}}\circ \mathcal T (x,y)$.  But,  $\mathcal T_{U^{j-1}}\circ \mathcal T (x,y) \in \mathcal D$ and hence while $(\mathcal T_{AC})^{-1} \circ \mathcal G(x,y)
\in \cup_{i=0}^{n-2}\; \mathcal T_{(AC^2)^i}(\mathcal D)$,  the  application of $\mathcal T_{AC}$ must send this value outside of that union.   That is, here also $\mathcal G(x,y)$ must  belong to $\Gamma$.
 
 The bijectivity of $\mathcal G$ on $\Gamma$ now follows immediately from that of $\mathcal T$ on $\Omega$.  
\end{proof}  

%------------------------------------------------------------------------ 
 \begin{Cor}\label{c:alpOneAccErg}\label{c:accErg}     Fix $n\ge 3$ and let $g(x)$ be  induced from $T_{n,1}$ and let $\Gamma$ be  as above.   Then $g(x)$ is expansive and is ergodic with respect to the probability measure that is the normalized marginal measure from $\mu = (1+xy)^{-2}\, dx\, dy$ on $\Gamma$.   
 \end{Cor}
%-----------------------------------------------------------------------
\begin{proof} From the definition of $\Omega$, it follows that $\Omega$ meets the curve $y= -1/x$ exactly in the $\mathcal T$-orbit of $(t, -1/t)$, which is of course a point in $\mathcal D$. % that is, in $\{ (r_i, -1/r_i \}_{0\le i \le n-2}$.     
As well,  the remainder of $\Omega$ lies above this curve (with $x\ge 0$).     Since $\Gamma  =  \Omega \setminus  \cup_{i=0}^{n-2}\; \mathcal T_{(AC^2)^i}(\mathcal D)$, it follows that $\Gamma$ not only does not meet the curve, but  in fact stays a bounded distance away.  From this,  hypothesis (b) of Theorem~\ref{t:evenExpanErgoNaturally} is satisfied for the piecewise M\"obius map $g(x)$  on the interval $(0, \epsilon_0)$ with $\mathcal G$ bijective on $\Gamma$;  the other hypotheses are easily verified,  and hence the result holds. 
\end{proof}

%------------------------------------------------------------------------
\section{Quilting as a proof tool}\label{s:quilting}    We now discuss a technique  used to date for solving for the planar extension of a piecewise M\"obius interval map beginning with such a planar extension for a sufficiently ``nearby" map.   This technique, called {\em quilting}, was introduced in \cite{KraaikampSchmidtSmeets},  and has its roots  in the discussion of the two-dimensional interpretation of ``insertion" and ``deletion" in the Ph.D. dissertation \cite{Kraaikamp}.    Theorem~\ref{t:finQuiltIsFine}, shows that one can use quilting to prove that certain properties are 
shared between appropriately nearby systems.

\subsection{Quilting defined, main properties announced}\label{ss:quilt}   
 We give a basic definition.

%------------------------------------------------------------------------ 
  \begin{Def}  \label{d: finitelyQuilted}   Suppose that $f, g$ are piecewise M\"obius interval maps on $\mathbb I_f, \mathbb I_g$ each with uncountably many cylinders and with finite nonzero $\mu$-measure planar two-dimensional domains of bijectivity $\Omega_f, \Omega_g$ for corresponding two-dimensional maps $\mathcal F, \mathcal G$, respectively.  
  For $x \in \mathbb I_f$ let $b_f(x)$ denote the $f$-digit of $x$ (informally, this thus denotes the corresponding M\"obius transformation which applied to $x$ gives the value $f(x)\,$), and similarly for $b_g(x)$.    
 Let  
 \[
 \begin{aligned} 
     \Delta &= \Delta_{f,g} = \{ x \in  \mathbb I_f \cap \mathbb I_g \mid b_f(x) \neq b_g(x)\}\; \text{and} \\
      \mathcal C &=  \mathcal C_{f} = \{ (x,y) \in \Omega_f \mid x \in  \Delta\}.
 \end{aligned}
 \]      
 \end{Def}
%------------------------------------------------------------------------ 

  We then construct a domain on which we will show that $\mathcal G$ is bijective (up to sets of measure zero) by deleting the forward $\mathcal F$-orbit of $\mathcal C$ and adding in the forward $\mathcal G$-orbit of $\mathcal C$ (here we extend $\mathcal G$ to be the piecewise map on $\mathbb I_g \times\mathbb R$ given by the $\mathcal T_M$ where $g$ is given piecewise by $x \mapsto M\cdot x$, recall \eqref{e:2DlocMap}).   In general, each of these orbits is infinite and might even sweep out the respective domains up to measure zero.   For a practical version of this approach, we introduce some finiteness conditions.  % (The first condition below is always fulfilled in our setting.) 

Recall that we are interested in measure theoretic results, and thus use disjointness of sets to mean that they meet in at most a null set.    See Figures~\ref{f:smallAlpQuilt} and  ~\ref{f:quiltLargeAlpLessThanDelta} for representations of  quilting in specific settings.
  
%------------------------------------------------------------------------ 
  \begin{Def}\label{d:generalQuilting}  
    We say that   $\Omega_g$ can be {\em countably quilted} from  $\Omega_f$ if $\mathcal C$ has positive measure and 
 \begin{enumerate} 
 \item[i.)]  There is an at most countable  partition of $\mathcal C$ by $\mathcal C_i$  and corresponding integers $d_i, a_i$ such that 
$\mathcal F^{1+d_i}_{\mid_{\mathcal C_i}}= \mathcal G^{1+a_i}_{\mid_{\mathcal C_i}}$;
\item[ii.)] The $\mathcal F^j(\, \mathcal C_i\,)$ indexed over all $i$ and  $1\le j \le d_i$  is pairwise disjoint and their union has strictly less than full measure in $\Omega_f$,   and similarly the $\mathcal G^j(\mathcal C_i)$  are pairwise disjoint, with their union of  finite measure; and, 
\item [iii.)]     \begin{equation}\label{eq:disjointDecomp}     \Omega_g = \bigg(\, \Omega_f \setminus \coprod_{i=1}^{\infty}\, \coprod_{j=1}^{d_i}\,\mathcal F^j(\,\mathcal C_i\,)\,\bigg) \amalg\;  \coprod_{i=1}^{\infty}\, \coprod_{j=1}^{a_i}\,\mathcal G^j(\mathcal C_i)\,.
\end{equation}
\end{enumerate}    
 \end{Def}
%------------------------------------------------------------------------   
Of course,   when the partition is finite  of cardinality $n$, then $n$ replaces the infinite upper limits appearing in \eqref{eq:disjointDecomp}.   We then speak of  {\em finite quilting}.   For simplicity's sake, we will write {\em quilting} to denote countable quilting. 

%--------------------------------------
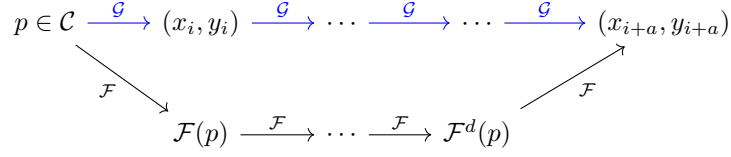
\begin{figure}
%\scalebox{.5}{
\begin{tikzcd}[column sep=2pc,row sep=2pc]
%--- a  -----------------------------------------------------------------------
p\in \mathcal C \arrow[blue]{r}{\mathcal G}\arrow[swap, ""{name=left2}]{rd}{\mathcal F}&(x_i, y_i)\arrow[blue]{r}{\mathcal G}&\cdots \arrow[blue]{r}{\mathcal G}&\cdots \arrow[blue]{r}{\mathcal G}&(x_{i+a}, y_{i+a}) &{}&\\
{}&\mathcal F(p)\arrow{r}{\mathcal F}&\cdots \arrow{r}{\mathcal F}&\mathcal F^d(p)\arrow[swap, ""{name=right2}]{ur}{\mathcal F}&{}&{} &%\arrow[phantom,from=left2,to=right2\\
%{} &{} & {} &{}& {} &{}&\\
\end{tikzcd}
%}
\caption{When quilting, forward $\mathcal G$-orbit segments beginning at a point in $\mathcal C$ rejoin forward $\mathcal F$-orbit segments.}
\label{f:orbConfigOfQuiltingOneType}%
\end{figure}
%------------------------------------------------------------------------

%------------------------------------------------------------------------ 
 \begin{Thm}\label{t:finQuiltIsFine}  Assume that $\Omega_f$ is of finite $\mu$-measure.  Quilting 
 preserves the properties of:  (a) ergodicity of two dimensional maps;  (b) this planar extension giving the natural extension of the interval map's system;  and, when (b) holds allows for an explicit expression of the entropy of $g$ in terms of that of $f$.   
\end{Thm}
%-----------------------------------------------------------------------

We give precise statements in the propositions of the subsequent subsection, which together prove the theorem.

%------------------------------------------------------------------------ 
\begin{Rmk}\label{rmk:iso}  Theorem~\ref{t:finQuiltIsFine} certainly holds when the systems of $\mathcal F$ and of $\mathcal G$ are isomorphic.   This is the case in particular  when   $a_i = d_i$ holds for each $\mathcal C_i$ of the quilting partition of $\mathcal C$.  Indeed,  we can then give an explicit isomorphism, $\varphi:\Omega_f \to \Omega_g$ by fixing $\Omega_f \setminus \coprod_{i=1}^{\infty}\, \coprod_{j=1}^{d_i}\,\mathcal F^j(\,\mathcal C_i)$ and  applying  $\mathcal F^j(x,y) \mapsto \mathcal G^j(x,y)$ for $1\le j \le a_i$ for each $(x,y) \in \mathcal C_i$,  for each $\mathcal C_i$.
\end{Rmk}
%------------------------------------------------------------------------

\subsection{Proofs of main properties}\label{ss:proofProperties}   

%------------------------------------------------------------------------ 
 \begin{Prop}\label{p:quiltErgo}   Suppose that $f, g$ are piecewise M\"obius interval maps such that $\Omega_g$ can be quilted from  $\Omega_f$,  and that $\mu(\Omega_f)< \infty$.  If $\mathcal F: \Omega_f \to \Omega_f$ is ergodic with respect to the measure $\mu$, then both $\mathcal G$ and $g$ are ergodic, with respect to 
$\mu$ on $\Omega_g$ and its marginal measure on $\mathbb I_g$, respectively.
\end{Prop}
%------------------------------------------------------------------------ 
\begin{proof} Since the dynamical system of $g$ is a factor of that of $\mathcal G$,  its ergodicity will follow from that of this latter system.     
Now suppose $E \subset \Omega_g$ is $\mu$-measurable,  is not of full measure,  and is invariant under $\mathcal G$. We aim to show that $E$ is a null set.     Recall that both $\mathcal F, \mathcal G$ preserve the measure. 

By  definition of $\mathcal C$,  we have that $\mathcal G$ agrees with $\mathcal F$ on  $\Omega_g \cap (\Omega_f \setminus \mathcal C)$.    The ergodicity of $\mathcal F$ shows that we may assume that each $\mathcal G$-orbit contained in $E$ which meets  $\Omega_g \cap (\Omega_f \setminus \mathcal C)$  also exits this set.    That is, we may assume that $E$ is contained in the set of forward $\mathcal G$-orbits of points of $\mathcal C$.    Each $p \in \mathcal C$  has an initial $\mathcal G$-orbit segment which  meets its forward $\mathcal F$-orbit.  By choosing the interpolating forward $\mathcal F$-orbit segments until their meeting and then through the next entrance to $\mathcal C$,  we form  an $\mathcal F$-invariant  set.    The ergodicity of $\mathcal F$  shows that this  is a nullset and therefore so must be $E$.      
\end{proof}

%------------------------------------------------------------------------ 
 \begin{Prop}\label{p:QuiltNaturally}   Suppose that $f, g$ are piecewise M\"obius interval maps such that $\Omega_g$ can be  quilted from  $\Omega_f$  and that the dynamical system of $\mathcal F$ is the natural extension of that of $f$. Then  the dynamical system of $\mathcal G$ is the natural extension of that of $g$. 
\end{Prop}
%------------------------------------------------------------------------ 
\begin{proof} 
 The natural extension is the minimal invertible system of which our given system is a factor.   Let us call any bi-infinite sequence $(x_i)_{i \in \mathbb Z}$ with each $x_i \in \mathbb I_g$ satisfying that for all $i$, $g(x_i) = x_{i+1}$  a {\em bi-infinite $g$-orbit} and similarly for our other maps.    We then say that a bi-infinite $\mathcal G$-orbit $(x_i, y_i)_{i \in \mathbb Z}$ projects to the bi-infinite $g$-orbit $(x_i)_{i \in \mathbb Z}$\,.
 
 To show that the system of $\mathcal G$ is the natural extension of the system of $g$, it suffices to show that for each bi-infinite $g$-orbit there is a unique   bi-infinite $\mathcal G$-orbit $(x_i, y_i)_{i \in \mathbb Z}$ in $\Omega_g$ projecting to it.      By hypothesis, the analogous statement is true for the pair $f, \mathcal F$.

  {\bf Step 1:  From $\mathcal G$- to  $f$-orbits.} 
We first associate to each  bi-infinite  $\mathcal G$-orbit $(x_i, y_i)_{i \in \mathbb Z}$ in $\Omega_g$ a bi-infinite $f$-orbit.     Let us discern two types of $\mathcal G$-orbit  orbit segments, those of type $a$, which begin at a point $p \in \mathcal C$ and end at the common point $\mathcal F^{d+1}(p) = \mathcal G^{a+1}(p)$; and those of type $c$, common to both by lying in $\Omega_g\cap (\Omega_f\setminus \mathcal C)$.   An orbit segment of type $c$ but whose extension thereafter exits $\Omega_g\cap (\Omega_f\setminus \mathcal C)$ must be such that the extension meets $\mathcal C$.  That is, this orbit segment is then followed by an orbit segment of type $a$.    Similarly,  an orbit segment of each of type $a$ is followed by a segment of one of two types.     In each of these settings, there is a uniquely associated $\mathcal F$-orbit segment;  this segment is given by equality when the $\mathcal G$-orbit segment is of type $c$,  and otherwise by the $\mathcal F$-orbit segment demanded as part of the definition of the type $a$ $\mathcal G$-orbit segment.    Further note that an orbit segment of either of the two types begins at the ending of a segment of one of the two types.   That is, we can uniquely extend the transcription from $\mathcal G$-orbit segments to $\mathcal F$-orbit segments infinitely in both directions.  With this, we have uniquely associated to each   bi-infinite $\mathcal G$-orbit $(x_i, y_i)_{i \in \mathbb Z}$ in $\Omega_g$ a  bi-infinite $\mathcal F$-orbit in $\Omega_f$.   Finally, by projecting, we find a bi-infinite $f$-orbit.   
 
 {\bf Step 2:  Uniqueness.}    Fix now a bi-infinite $g$-orbit $\gamma=(x_i)_{i \in \mathbb Z}$.  Let us call  any choice of  $\mathcal G$-orbit in $\Omega_g$ projecting to $\gamma$ a `lift' of $\gamma$.  
 From the above, to a bi-infinite $g$-orbit $\gamma=(x_i)_{i \in \mathbb Z}$  we associate  one bi-infinite $f$-orbit  per lift of $\gamma$.    We now aim to show that these bi-infinite $f$-orbits are one and the same.
 
We first note that $\gamma$ determines a  unique bi-infinite sequence $(M_i)_{i\in \mathbb Z}$ such that $x_{i+1} = M_i\cdot x_i$;  and hence any choice of bi-infinite $\mathcal G$-orbit in $\Omega_g$ projecting to $\gamma$ satisfies  $(x_{i+1}, y_{i+1}) = \mathcal T_{M_i}(x_i,y_i)$ for all $i$.   The analogous statement holds for bi-infinite $f$- and $\mathcal F$-orbits.

Now, any  lift of $\gamma$ can be partitioned into segments of type $a$ and $c$.   Consider the set $\mathcal J = \{j \, \vert\, x_j \in \Delta\}$.   If $\mathcal J = \emptyset$  then any $\mathcal G$-orbit in $\Omega_g$ projecting to $\gamma$ lies in $\Omega_g\cap (\Omega_f\setminus \mathcal C)$.   This is then also an  $\mathcal F$-orbit.  The projections agree, and thus $\gamma$ itself is our desired bi-infinite $f$-orbit.  Now if  $\mathcal J$ is non-empty  then any lift of $\gamma$  is such that there is some $j$ with $p = (x_j, y_j) \in \mathcal C$ and the previous paragraph  then shows that this lift's associated bi-infinite $f$-orbit contains the forward $f$-orbit of $x_j$.     If another lift of $\gamma$ is such that it contains a point $q = (x_j, y'_j)$ with $q \notin \mathcal C$ then $q$ occurs in what we can call the `middle' of some $\mathcal G$-orbit segment of type $a$.    But then this segment is announced by a 
$p' = (x_{j'}, y_{j'})\in \mathcal C$ with $j'<j$,  and again by invoking the previous paragraph, we have that the forward $f$-orbit of 
$x_{j'}$ is contained in this lift's associated bi-infinite $f$-orbit.  Furthermore, this forward $f$-orbit contains that of $x_j$.   From this, if $\mathcal J$ has a least element $j$, then every lift of $\gamma$ has its associated bi-infinite $f$-orbit containing the forward $f$-orbit of the corresponding $x_j$.   But,  each such lift must then have backwards infinite orbit completely of type $c$.  That is, the associated bi-infinite $f$-orbits must also all agree for indices less than $j$.  By similar reasoning, in the case of $\mathcal J$ having no least element, all of the lifts of $\gamma$ share the same bi-infinite $f$-orbit.    
 
 {\bf Step 3: Conclusion.}    By hypothesis,   each bi-infinite $f$-orbit  has exactly one bi-infinite $\mathcal F$-orbit in $\Omega_f$ projecting to it.   Since the bi-infinite $\mathcal F$-orbit in $\Omega_f$ associated to a bi-infinite $\mathcal G$-orbit in $\Omega_f$ certainly uniquely identifies this $\mathcal G$-orbit,  we conclude that 
$\gamma$ has exactly one bi-infinite $\mathcal G$-orbit in $\Omega_f$ projecting to it.
 \end{proof}

%------------------------------------------------------------------------ 
 \begin{Prop}\label{p:quiltEntropy}   Suppose that $f, g$ are piecewise M\"obius interval maps such that $\Omega_g$ can be  quilted from  $\Omega_f$.  Then the entropy of $\mathcal F$ and $\mathcal G$ are related by 
\[ h(\mathcal G) =  h(\mathcal F) \; \dfrac{ \mu( \Omega_f)}{\mu(\Omega_{g})}\,.\]
If furthermore the dynamical system of $\mathcal F$ is the natural extension of that of $f$,  then  
 the entropy of $g$ is given by  
 \begin{equation} \label{e:genQuiltEntropy}
 h(g) = \bigg(1 + \sum_{i=1}^{\infty}\, (a_i - d_i) \, \nu(\,\Delta_i\,)\bigg)^{-1} h(f),
 \end{equation}
 where for each $i$, $\mathcal C_i$ projects to $\Delta_i \subset \mathbb I_f$ and $\nu$ is the marginal probability measure induced from $\mu$ on $\Omega_f$.
\end{Prop}
%------------------------------------------------------------------------ 
\begin{proof}  We can compute the entropy of $\mathcal G$ in terms of that of $\mathcal F$ by using Abramov's formula  (as mentioned in  \ref{ss:BasicsDynSys})  Let  $\Omega_{f,g} := \Omega_f \setminus \coprod_{i=1}^{\infty}\, \coprod_{j=1}^{d_i}\,\mathcal F^j(\,\mathcal C_i\,)$.   The first return maps induced from each of $\mathcal F$ and $\mathcal G$ to  $\Omega_{f,g}$ are equal.   Therefore,
\[ \dfrac{h(\mathcal G) \, \mu( \Omega_g)}{\mu(\Omega_{f,g})} =   \dfrac{h(\mathcal F) \, \mu( \Omega_f)}{\mu(\Omega_{f,g})}.\]
Since our various maps are $\mu$-measure preserving,   
\[ h(\mathcal G) = \bigg(1 + \sum_{i=1}^{\infty}\, (a_i - d_i) \, \dfrac{\mu(\mathcal C_i)}{\mu(\Omega_{\alpha})}\bigg)^{-1} h(\mathcal F).\]

If $\mathcal F$ gives the natural extension of  $f$, then they have the same entropy.  Furthermore,  from the previous proposition,   $\mathcal G$ then gives the natural extension of $g$ and thus these also have equal entropy.  Since each $\Delta_i\,$ is the projection of  $\mathcal C_i$, we have that $\nu(\,\Delta_i\,)= \mu(\mathcal C_i)/\mu(\Omega_{\alpha})$.   Therefore, \eqref{e:genQuiltEntropy} holds.   
\end{proof} 

\subsection{Property of realizable first return type is also preserved}\label{ss:quiltRealFirstReturn}    We use notation,  terminology and results presented in \S~\ref{ss:ArnouxMethod}.

%------------------------------------------------------------------------ 
 \begin{Prop}\label{p:quiltFirstReturn}   Suppose that $f, g$ are expansive piecewise M\"obius interval maps such that $\Omega_g$ can be quilted from  $\Omega_f$.      Suppose that $\Gamma_f = \Gamma_g$ is of finite covolume, $\widehat{\Gamma}_f = \widehat{\Gamma}_g$, and   $f$ is of  realizable first return type.   Then $g$ is also of  realizable first return type.  Furthermore,  both maps:  are ergodic;  have their planar extensions as natural extensions;  and, are 
factors of the first return map to a section for the geodesic flow on $T^1( \Gamma_f \backslash \mathbb H)$. 
\end{Prop}
%------------------------------------------------------------------------ 

\begin{proof}    Since $f$ is of realizable first return type, from Theorem~\ref{t:detOneSetting}  or Theorem~\ref{t:detPlusMinOneSetting},  we have that $f$ is ergodic,  expansive,  that its planar extension gives a natural extension, and that $w_f h(f) \mu(\Omega_f) =  \text{vol}(\,   T^1( \Gamma_f\backslash \mathbb H)\, )$, where $w_f \in \{1,2\}$ is equal to $1$ if only if $\widehat{\Gamma}_f =  \Gamma_f$.

     Now,  Proposition~\ref{p:quiltEntropy}   gives $h(\mathcal G)  \mu(\Omega_{g}) =  h(\mathcal F)  \mu( \Omega_f)$.  Proposition~\ref{p:QuiltNaturally} and the fact that entropy is shared by a map and its natural extension, then  gives that $h(g)  \mu(\Omega_{g}) =  h(f)  \mu( \Omega_f)$. Hence,   $h(g)  \mu(\Omega_{g})$ equals $w_f$ times the volume of the unit tangent bundle of  $\Gamma_g \backslash \mathbb H$.     By hypothesis, $g$ is expansive, Proposition~\ref{p:quiltErgo} shows $g$ is ergodic, and hence $g$ is of  realizable first return type.
\end{proof} 

   We now also use terminology and results of \S~\ref{ssBern}. Recall that `quasi-isomorphic' means having isomorphic natural extensions.    
%------------------------------------------------------------------------ 
 \begin{Prop}\label{p:quiltBernoulli}   Assume the hypotheses of the previous proposition.  
If  $\vert 1/f'\vert$ is of bounded variation then  the natural extension of $f$ is a Bernoulli system, similarly for $g$; if both have this property, then their systems are quasi-isomorphic if and only of they share the same entropy value.    
\end{Prop}
%------------------------------------------------------------------------ 
\begin{proof}   The fulfillment of the bounded variation condition for either map completes the hypotheses for the Rychlik result, Theorem~\ref{t:rychlikBernoulli}, and thus guarantees that the natural extension system is Bernoulli.   If this occurs for both maps, Ornstein's fundamental result  Theorem~\ref{t:Ornstein} shows that isomorphism of the natural extension systems is determined by entropy values. 
\end{proof}

%------------------------------------------------------------------------ 
\begin{Rmk}\label{rmk:crossBern}  In the above, if $\vert 1/f'\vert$ is of bounded variation and $f$  is of determinant one type (that is, if $\hat{\Gamma}_f = \Gamma_f$) then the cross section to the geodesic flow on $T^1( \Gamma_f\backslash \mathbb H)$ associated to $f$ in Theorem~\ref{t:detOneSetting}, being a version of the natural extension,  is Bernoulli.  It is not always true that a cross section to a Bernoulli flow has Bernoulli first return system, as discussed in \cite{OrnsteinWeiss2} and \cite{OrnsteinRudolphWeiss}.    
\end{Rmk}
%------------------------------------------------------------------------

%------------------------------------------------------------------------ 
\begin{Rmk}\label{rmk:crossBernHoldsForCKSsys}   For each of the maps $T = T_{n,\alpha}$  discussed in  \S~\ref{ss:CKS} and for each $x$ in its domain, one has $T'(x)$ being equal to either $(C\cdot x)'$ or $(C^2\cdot x)'$ where $C$ is given in \eqref{e:generators}.    It easily follows that $\vert 1/T'\vert$ is of bounded variation.   Results of \cite{CaltaKraaikampSchmidt} and the determination in \cite{CaltaKraaikampSchmidtPlanar}  of planar extension systems then allow one to invoke the results of this subsection for the  $T_{n,\alpha}$.
\end{Rmk}
%------------------------------------------------------------------------

\subsection{Finite quilting for close neighbors that match}\label{ss:FiniteQuiltingClose}  We quickly give basic definitions which capture the essence of the {\em matching} interval phenomenon  --- also known as:  synchronization \cite{CaltaKraaikampSchmidt}, or short cycles \cite{KatokUgarcoviciStructure} ---   studied in various of families of continued fraction type maps.   Thereafter we show that under reasonable assumptions upon matching intervals there are subintervals on which  quilting applies.     Our terminology and notation attempts to negotiate between that of \cite{CT} and of \cite{CaltaKraaikampSchmidt}.    

\subsubsection{Matching:  relations,  intervals and their exponents}\label{sss:matching}
Suppose that we are given a one parameter family of piecewise M\"obius interval functions, $\{T_{\alpha}\}_{\alpha \in \mathcal I}\,$, indexed by $\alpha$ ranging over some real interval $\mathcal I$, each of    whose interval of definition   $\mathbb I_{\alpha} = [\, \ell_0(\alpha), r_0(\alpha)\,)$ is of fixed length  $\lambda$.  (We will always assume that also  $\vert\, \mathcal I \,\vert \le \lambda$.) In particular, letting  
\begin{equation}\label{e:nameShift} S: x \mapsto x + \lambda,
\end{equation}
for all $\alpha$ we have  $\ell_0(\alpha) = S^{-1}\cdot r_0(\alpha)$.   

A {\em matching interval} $J \subset \mathcal I$ is a subinterval such (1) that there are some $m,n \in \mathbb N$ such that for all $\alpha \in J$ we have   $T_{\alpha}^m(\, \ell_0(\alpha)\,) = T_{\alpha}^n(\, r_0(\alpha)\,)$,  where  $m,n$ are minimal except possibly at finitely many $\alpha \in J$ --- we call the $\alpha \in J$ where this minimality holds  {\em typical} --- ; (2) the digits of the expansions of the endpoints $\ell_0(\alpha)$ agree in that for all $1\le i < m$ there is a M\"obius transformation $M_i$ such that $T_{\alpha}^i(\, \ell_0(\alpha)\,) = M_i\cdot \ell_0(\alpha)$  for all $\alpha \in J$, and similarly for the $r_0(\alpha)$; 
and,  (3) there are M\"obius transformations $L_J, R_J$ such that for all $\alpha \in J$ we have $T_{\alpha}^{n-1}(\, r_0(\alpha)\,)= R_J\cdot r_0(\alpha)$ and    $T_{\alpha}^{m-1}(\, \ell_0(\alpha)\,) = L_JS^{-1}\cdot r_0(\alpha)$.   We call $m, n$ the {\em matching exponents} of $J$.  In fact, we need a further property of the family: say that the family has a {\em matching relation} if (4) there is some M\"obius transformation $M$ such that  for any matching interval,   $M L_J S^{-1}= R_J$.

%------------------------------------------------------------------------ 
\begin{Rmk}\label{rmk:moreMatching}   Note that the notion of matching is easily extended to orbits to the left and right of points of discontinuities, see Bruin {\em et al} \cite{BruinEtAl} where this is done for a related setting.  We forgo doing this here, for simplicity's sake.
\end{Rmk}
%------------------------------------------------------------------------ 

Many of the well-studied families of continued fractions have matching relations. 
 %------------------------------------------------------------------------ 
\begin{Eg}\label{eg:goodFamiliesHaveMatching}   We briefly indicate a few of these.

\begin{itemize} 
\item     The Nakada $\alpha$-continued fractions has a matching relation:  let 
\begin{equation}\label{e:wmat} W = \begin{pmatrix}1&0\\-1&-1  \end{pmatrix},
\end{equation}
 then combining (\cite{KraaikampSchmidtSteiner} Remark~6.9) with (\cite{KraaikampSchmidtSteiner} Lemmas ~6.2, 6.4) shows that $M=W$ gives the  relation  for each matching interval $J$. (Note that   \cite{CT}  also showed that there are matching intervals in the Nakada family, but express the matching relations in a different manner than here.)\\

\item   In the setting of $\alpha$-continued fraction expansions with odd partial quotients,  \cite{HKLM}   show a matching relation of the form $M L_J S^{-1}= R_J$, with $M = \begin{pmatrix}-1&0\\2&1  \end{pmatrix}$ see the final line on their p.~28.\\

\item The countably many families of  \S~\ref{ss:CKS} coming from \cite{CaltaKraaikampSchmidt},  are such that for each $n$ the corresponding family has a matching relation for  the small $\alpha$, those with $\alpha \in (0,\gamma_{n})$, see (\cite{CaltaKraaikampSchmidt}, Prop.~5.2).    
There  is a distinct matching relation on $(\gamma_{n}, 1]$    (more precisely for those parameter values,   one splits each matching interval $J$ into two pieces and finds that there is a matching relation for all of the left hand pieces and a nearly identical relation for all of the right hand pieces); see (\cite{CaltaKraaikampSchmidt}, Lemma~6.2).   
\end{itemize}  

\end{Eg}
%------------------------------------------------------------------------ 
 
 \subsubsection{Close neighbors}

 We are  interested in applying quilting when $\alpha, \alpha'$ lie within the same synchronization interval;  quilting succeeds in the most straightforward manner if we require that $\alpha, \alpha'$ are particularly close.  
 
To lighten notation, let us write $\ell_i$ and $\ell'_i$ for each of $T_{\alpha}^i(\ell_0(\alpha)\,)$ and $T_{\alpha'}^i(\ell_0(\alpha')\,)$, respectively and similarly for the orbits of the other endpoints.   We use the notation of Definition~\ref{d: finitelyQuilted}  in the following.
%------------------------------------------------------------------------ 
  \begin{Def} 
  \label{d: tightSyn} Suppose that $J$ is a matching interval with corresponding matching exponents $m,n$.  We say that $\alpha, \alpha' \in J$  are {\em close neighbors} if
  $\ell'_i, \ell_i, r'_j, r_j \in  \mathbb I_{\alpha'} \cap \mathbb I_{\alpha}$ for all  $1\le i   \le  m$ and $1 \le j   \le n$.  
 \end{Def}
%------------------------------------------------------------------------  
\noindent   Note that the hypothesis on the orbit entries can equivalently be written as:   The $\alpha$-digit of $\ell'_i$ equals the $\alpha'$-digit of $\ell_i$ and vice versa for each $1\le i <m$, and similarly for the various $r_j, r'_j$.\\

  The majority of the aforementioned families have the following properties.  
%------------------------------------------------------------------------ 
  \begin{Def}  \label{d:shiftDigs}  Fix a family of piecewise M\"obius interval maps, $\mathcal F = \{ T_{\alpha} \,\vert\, \alpha \in \mathcal I\}$.

 \begin{enumerate} 
  \item   We say that $\mathcal F$  is {\em of purely shift digit changes} if for any $\alpha, \alpha' \in\mathcal I$ whenever $x \in \Delta_{T_{\alpha}, T_{\alpha'}}$ then  $T_{\alpha}(x) = S^{\pm 1} T_{\alpha'}(x)$, where $S$ is as in \eqref{e:nameShift}.\\

  \item Suppose that $\alpha \in J$ with $J$ a matching interval of matching exponents $m,n$.  We say that $\Omega_{\alpha}$ has {\em locally constant fibers} if its vertical fibers are constant between  the points of the initial orbits of $\ell_0(\alpha)$ and $r_0(\alpha)$:   that is, if the fibers are constant above the connected components  of  the complement in $\mathbb I_{\alpha}$ of   $\{ \ell_i(\alpha)\,\vert\, 0\le i \le m-1\} \cup \{ r_j(\alpha)\,\vert \, 0\le j \le n-1\}$.
 \end{enumerate}
\end{Def}
%------------------------------------------------------------------------  

%--------------------------------- Figure quilt from Omega_{0.14} to Omega_{0.135} Of v=1, k=1----------------------------------------------
\begin{figure}[h]
\scalebox{.75}{
\begin{tikzpicture}[x=5cm,y=5cm] 
\draw  (-1.72, -0.4)--(-0.55, -0.4); 
\draw  (-0.55, -0.4)--(-0.55, -0.2); 
\draw  (-0.55, -0.2)--(0.28, -0.2); 
\draw  (0.28,  1.79)--(-0.25,  1.79); 
\draw  (-0.25,  1.79)--(-0.25,  1.64); 
\draw  (-0.25,  1.64)--(-0.44,  1.64); 
\draw  (-0.44,  1.64)--(-0.44,  0.74); 
\draw  (-0.44,  0.74)--(-0.74,  0.74);
\draw  (-0.74,  0.74)--(-0.74,  0.44);
\draw  (-0.74,  0.44)--(-1.64,  0.44);
\draw  (-1.64,  0.44)--(-1.64,  0.39);   
\draw  (-1.64,  0.39)--(-1.72,  0.39);  
\draw  (-1.72,  0.39)--(-1.72,  -0.4); 
        %split at x=0 %
\draw  (0.0, -0.2)--(0.0, 1.79);         
        %mark regions of Delta(-1,1) and Delta(-2,1)
        % just eyeballing from mms and guessing values
        % \setdashes 
\draw[thin,dashed] ( -0.84, -0.4)--(-0.84, 0.44);        
\draw[thin,dashed] (-0.32, 1.64)--(-0.32, -0.2); 
\draw[thin,dashed] (-0.19, 1.79)--(-0.19, -0.2); 
         % Label these
\node at (-1, 0) {$-1$};  
\node at (-0.55, 0) {$-2$};      
\node at (-0.26, 0) {$-3$};      
\node at (-0.1, 0) {$\cdots$};  
\node at (0.13, 0) {$\cdots$}; 
         %Label vertices 
 \node at (-1.7, -0.5) {$(\ell_0, y_{-2})$}; 
 \node at (-1.95,  0.40) {$(\ell'_0, y_1)$}; 
 \node at ( -1.5, 0.55) {$(\ell_3, y_2)$}; 
 \node at (-0.82, 0.82) {$(\ell'_2, y_3)$};
 \node at (-0.6, 1.65)  {$(\ell_1, y_4)$}; 
 \node at (-0.3, 1.9)  {$(\ell'_4, y_5)$}; 
 \node at (0.5, -0.2) {$(r_0, y_{-1})$}; 
 \node at (0.35, 1.9) {$(r'_0, y_5)$}; 
 \node at (-0.5, -0.5)  {$(r_1, y_{-2})$}; 
      %label x=0
 \node at (0, -0.2)  {$0$}; 
           %Place dot at each vertex
 \foreach \x/\y in {-1.72/-0.4,-0.55/-0.4, 
 0.28/-0.2,  0.24 /1.79, -0.28/1.79, -0.44/1.64, -0.77/0.74,
-1.64/0.44, -1.76/0.39%
} { \node at (\x,\y) {$\bullet$}; } 
%--- and now the quilting ---
%first delete  2;  
\draw[pattern=north east lines, pattern color=red]    (0.24, -0.2)--(0.24, 1.79) -- (0.27, 1.79) --(0.27, -0.2) -- cycle; 
\draw[pattern=north east lines, pattern color=red]     (-0.55, -0.4)--(-0.55, -0.2) --  (-0.59, -0.2)--(-0.59, -0.4) -- cycle; 
%add 5
\draw[pattern=north west lines, pattern color=blue]    (-1.72, -0.4)--(-1.76, -0.4) -- (-1.76, 0.39) --(-1.72, 0.39) -- cycle; 
\draw[pattern=north west lines, pattern color=blue]   (-1.64,  0.44)--(-1.64,  0.39) -- (-1.67,  0.39)--(-1.67,  0.44) -- cycle; 
\draw[pattern=north west lines, pattern color=blue]    (-0.74,  0.74)--(-0.74,  0.44) -- (-0.77,  0.44)--(-0.77,  0.74) -- cycle; 
\draw[pattern=north west lines, pattern color=blue]     (-0.44,  1.64)--(-0.44,  0.74) --  (-0.47,  0.74)--(-0.47,  1.64) -- cycle;  
\draw[pattern=north west lines, pattern color=blue]    (-0.25,  1.79)--(-0.25,  1.64) --  (-0.28,  1.64)--(-0.28,  1.79) -- cycle;  
%quilt square
\draw[pattern=north west lines, pattern color=blue]    (-1.0, 0.35 )--(-1.0, 0.3 ) --  (-1.1, 0.3 )--(-1.1, 0.35 ) -- cycle; 
\draw[pattern=north east lines, pattern color=red]    (-1.0, 0.35 )--(-1.0, 0.3 ) --  (-1.1, 0.3 )--(-1.1, 0.35 ) -- cycle; 
%pointers and labels to deletions
\node at (0.27,  0.8)[pin={[pin edge=<-, pin distance=12pt]0:{$\mathcal D := \mathcal T_{\alpha}(\mathcal C)$}}] {};
\node at (-0.55,  -0.3)[pin={[pin edge=<-, pin distance=12pt]0:{$\mathcal T_{\alpha}^{\,2}(\mathcal C)$}}] {};
%pointers and labels to additions
\node at (-1.76,  0)[pin={[pin edge=<-, pin distance=12pt]180:{$ \mathcal T_{A^{-1}}(\mathcal D) = \mathcal T_{\alpha'}(\mathcal C)$}}] {};
\node at (-1.67,  0.42)[pin={[pin edge=<-, pin distance=60pt]135:{$\mathcal T_{\alpha'}^{\,2}(\mathcal C)$}}] {};
\node at (-0.77,  0.58)[pin={[pin edge=<-, pin distance=80pt]160:{$\mathcal T^{\,3}_{\alpha'}(\mathcal C)$}}] {};
\node at (-0.47,  1.2)[pin={[pin edge=<-, pin distance=60pt]180:{$\mathcal T^{\,4}_{\alpha'}(\mathcal C)$}}] {};
\node at (-0.3,  1.74)[pin={[pin edge=<-, pin distance=80pt]180:{$\mathcal T^{\,5}_{\alpha'}(\mathcal C)$}}] {};
%pointers and labels to patch
\node at (-1.1,  0.3)[pin={[pin edge=<-, pin distance=5pt]250:{$\mathcal T^{\,6}_{\alpha'}(\mathcal C) = \mathcal T^{\,3}_{\alpha}(\mathcal C)$}}] {};
\end{tikzpicture} 
}
\caption{{\bf Quilting from a close neighbor.}  Quilting in the setting of `small' $\alpha$ of systems discussed in \S~\ref{ss:CKS}.  Here $n=3, \alpha = 0.14, \alpha' = 0.135$, and quilting given $\Omega_{\alpha}$ results in $\Omega_{\alpha'}$.   Domain $\Omega_{\alpha}$,  details of which are in \cite{CaltaKraaikampSchmidtPlanar}, not drawn fully to scale.   Integers $ -1, -2, -3$ indicate regions fibering over cylinders of corresponding `simplified digits'.   The forward $\mathcal T_{\alpha}$-orbit of $\mathcal C$ is deleted, while the forward  $\mathcal T_{\alpha'}$-orbit of $\mathcal C$ is added, until the ``hole" created by excising  $\mathcal T_{\alpha}^{\,3}(\mathcal C)$ is ``patched" in by  $\mathcal T_{\alpha'}^{\,6}(\mathcal C)$.   See Proposition~\ref{p:closeNbhsQuiltTogether}. }%
\label{f:smallAlpQuilt}%
\end{figure}
%-------------------------------------------------------------------------------------------
 
 See Figure~\ref{f:smallAlpQuilt} for a special case illustrating the following.
%------------------------------------------------------------------------ 
 \begin{Prop}\label{p:closeNbhsQuiltTogether}  Suppose that $\mathcal F = \{ T_{\alpha} \,\vert\, \alpha \in \mathcal I\}$ is a family of piecewise M\"obius interval maps of purely shift digit changes. Suppose further that   $\alpha, \alpha'$ are close neighbors with $\alpha$  typical for their common matching interval.    
 Then $\Omega_{\alpha'}$ can be finitely quilted from $\Omega_{\alpha}$.    Furthermore,  $\Omega_{\alpha}$ has locally constant fibers if and only if $\Omega_{\alpha'}$  does. 
\end{Prop}
%-----------------------------------------------------------------------
\begin{proof} 
For ease, assume that $\alpha' < \alpha$, the other case follows by a symmetric argument. Thus, $\mathcal T_{\alpha}(\mathcal C)$ fibers over $[r'_0, r'_0)$ and hence we have  $\mathcal T_{\alpha'}(\, \mathcal C) = \mathcal T_{S^{-1}}\circ \mathcal T_{\alpha} (\, \mathcal C)$.   Since $\alpha, \alpha'$ are close neighbors, they share a common matching interval $J$, let   $m, n$ be its matching exponents. 
 Let $U = U_{\alpha}$ be the M\"obius transformation such that $\ell_m(\alpha) = U L_J\cdot  \ell_0(\alpha)$ and $V = V_{\alpha}$ be such that $r_n(\alpha) = V R_J\cdot r_0(\alpha)$.  
Since $\mathcal F$ has only purely shift digit changes,   the condition that $\ell'_i, \ell_i, r_j, r'_j  \in  \mathbb I_{\alpha'} \cap \mathbb I_{\alpha}$ for the various $i, j$  implies that  when $i<m$ and $j<n$ these  $\ell'_i, \ell_i, r_j, r'_j$ lie outside of $\Delta_{T_{\alpha}, T_{\alpha'}}$.  Hence,    $\mathcal T_{\alpha'}^{m+1}(\mathcal C) = \mathcal T_{U  L_J S^{-1}} \circ \mathcal T_{\alpha}(\mathcal C)  = \mathcal T_{V R_J} \circ \mathcal T_{\alpha}(\mathcal C) =  \mathcal T_{\alpha}^{n+1}(\mathcal C)$. 

We claim that  $\bigcup_{i=1}^{m} \, \mathcal T_{\alpha'}^{i}(\mathcal C)$ is disjoint from $\Omega_{\alpha}$.  To this end,  let $k'$ be minimal such that $\mathcal T_{\alpha'}^{k'}(\mathcal C) \cap \Omega_{\alpha}$ has positive measure.   (Since $\mathcal T_{S^{-1}}\circ \mathcal T_{\alpha} (\, \mathcal C)$ projects to $[\ell'_0, \ell_0)$ clearly $k' > 1$.)    Let $(x,y) \in \mathcal C$ such that $\mathcal T_{\alpha'}^{k'}(x,y)\in \Omega_{\alpha}$.    
Again the close neighbors property gives that thereafter the forward $\mathcal T_{\alpha'}$-orbit  of this point is given by $\alpha$-admissible M\"obius transformations, and thus agrees with its forward $\mathcal T_{\alpha}$-orbit until we reach the  $\mathcal T_{\alpha'}^{m+1}(x,y) =   \mathcal T_{\alpha}^{n+1}(x,y)$.  The bijectivity of $\mathcal T_{\alpha}$ then implies that there is some $k$ such that $\mathcal T_{\alpha}^k \circ \mathcal T_{\alpha'}^{k'} (x,y) =  \mathcal T_{\alpha}^{n+1}(x,y)$ and hence   $\mathcal T_{\alpha'}^{k'} (x,y) =  \mathcal T_{\alpha}^{n+1-k}(x,y)$.    If $k' < m+1$ then by the positivity of the measure of such points, we deduce that there are factorizations $L_J = L''_J U' L'_{J}$  and $R_J = R''_J V' R'_J$  with  $L''_J = R''_J$
and  $U' L'_J S^{-1} = V' R'_J$.     Since $\alpha'<\alpha$ are close neighbors, there are other values $\alpha'<\alpha'' < \alpha$ that are also close neighbors of $\alpha$ and hence we find that there is an interval $J' \subseteq J$ with matching exponents $m', n'$.  But, this contradicts the definition of $J$ as the matching interval for its typical $\alpha, \alpha'$.   Therefore, we must have that $k' = n+1$ and the disjointness of  $\bigcup_{i=1}^{m}\, \mathcal T_{\alpha'}^{i}(\mathcal C)$ from $\Omega_{\alpha}$ does hold. 

We next claim that  the $\mathcal T_{\alpha'}^{i}(\mathcal C)$ are pairwise disjoint.  To this end, suppose that $\mathcal T_{\alpha'}^{i}(\mathcal C)$ meets $\mathcal T_{\alpha'}^{j}(\mathcal C)$ in positive $\mu$-measure for some $1\le i \le j \le m+1$.   Then the same is true for $\mathcal T_{\alpha'}^{i+m+1-j}(\mathcal C)$ and $\mathcal T_{\alpha'}^{m+1}(\mathcal C) = \mathcal T_{\alpha}^{n+1}(\mathcal C)$ and arguing as above, we find that $i=j$.       The analogous argument shows that the $\mathcal T_{\alpha}^{j}(\mathcal C)$ are disjoint. 

 Finally, due to the disjointness properties which we have shown,  it follows that  $\Omega_{\alpha}$ has locally constant fibers if and only if $\Omega_{\alpha'}$  does.
\end{proof}

  We desire to prove the analog of the above theorem holds also for the setting of  `large' $\alpha, \alpha'$ as defined in \S~\ref{ss:CKS}.  In that setting, digit changes other than shifts can occur.   However, for close neighbors, the location of the second type of digit changes is constrained to a single interval and the digit change is completely explicit; for a hint of this,  see Figure~\ref{f:quiltLargeAlpLessThanDelta}.   In brief, the following is a mild extension of the previous result, but one which we call on in \cite{CaltaKraaikampSchmidtPlanar}.

Recall that the shift $S$ is given in \eqref{e:nameShift}. 
 
 %---------------------------------Figure quilting  large alpha on left side of $J_{-k,v}, alpha = 0.86 to \alpha' = 0.855 ----------------------------------------------
\begin{figure}[h]
\scalebox{.7}{
\begin{tikzpicture}[x=6cm,y=6cm] 
\draw  (-.28, -0.72)--(0.6, -0.72); 
\draw  (0.6, -0.72)--(0.6,  -0.44); 
\draw  (0.6, -0.44)--(1.1,  -0.44);   
\draw  (1.1, -0.44)--(1.1,  -0.39);   
\draw  (1.1, -0.39)--(1.72,  -0.39);   
\draw  (1.72, -0.39)--(1.72,  0.35); 
\draw  (1.72, 0.35)--(0.25,  0.35); 
\draw  (0.25, 0.35)--(0.25,  0.2);  
\draw  (0.25, 0.2)--(-.28,  0.2);  
\draw  (-.28, 0.2)--(-.28,  -0.72);
       %split at x=0, 1 % 
 \draw  (0.0, -0.72)--(0.0, .2); 
 \draw  (1, -0.44)--(1, .35); 
        %mark regions of Delta(-1,1) and Delta(-2,1)
        % just eyeballing from mms and guessing values
        %\setdashes 
\draw[thin,dashed] (-0.17, -0.72)--(-0.17, 0.2);    
\draw[thin,dashed] (0.19, -0.72)--(0.19, 0.2); 
\draw[thin,dashed] (0.33, -0.72)--(0.33, 0.35);      
\draw[thin,dashed] ( 0.78, -0.44)--( 0.78, 0.35);  %line of frak b 
\draw[thin,dashed] (0.92, -0.44)--(0.92, 0.35);  
\draw[thin,dashed] (1.2, -0.39)--(1.2, 0.35);  
\draw[thin,dashed] (1.4, -0.39)--(1.4, 0.35);  
\draw[thin,dashed] (1.08, -0.44)--(1.08, 0.35);  
        % Label these
%\node at (-.23, 0){\tiny{$(-2,1)$}};
\node at (-.23, 0)[pin={[pin edge=<-, pin distance=16pt]300:{\tiny{$(-2,1)$}}}] {};
\node at (0.55, -0.2) {\tiny{$(1,1)$}};  
\node at (0.25, 0) {\tiny{$(2,1)$}}; 
\node at (0.85, 0.12) {\tiny{$(-2,2)$}};      
\node at (0.10, 0) {\tiny{$\cdots$}}; 
\node at (-0.1, 0) {\tiny{$\cdots$}};   
\node at (0.96, -0.01) {\tiny{$\cdots$}}; 
\node at (1.05, -0.01) {\tiny{$\cdots$}}; 
\node at (1.14, 0) {\tiny{$(3,2)$}};  
\node at (1.3, 0) {\tiny{$(2,2)$}};   
\node at (1.55,  -0.2) {\tiny{$(1,2)$}};   
 %Label vertices 
\node at (-.3, -0.8) {$(\ell_0, y_{-3})$}; 
\node at (-.45, 0.2) {$(\ell'_0, y_1)$}; 
\node at (0.3,  0.45)  {$(\ell_1, y_2)$};  
\node at (0.62, -0.8)  {$(r_1, y_{-3})$}; 
\node at (1.2, -0.5)  {$(r_2, y_{-2})$}; 
%\node at (.78, 0.43)  {$\mathfrak b$};
\node at (1.7, -0.5)  {$(r'_0, y_{-1})$};  
\node at (1.75,0.45)  {$(r_0, y_{2})$};   
\node at (0,0)  {$0$};  
\node at (1, -0.5)  {$1$};  
          %Place dot at each vertex  (includ. some prime points
 \foreach \x/\y in {-.28/-0.72, 0.6/-0.72, 1.1/-0.44, 
 1.68/-0.39,  1.72/0.35, 0.25/0.35, -.3/0.2%
} { \node at (\x,\y) {$\bullet$}; } 
%--- and now the quilting ---
%first delete  2x2;  
\draw[pattern=north west lines, pattern color=orange]    (1.68, -0.3)--(1.68, 0.35) -- (1.72, 0.35)--(1.72,-0.3)-- cycle; 
\draw[pattern=north east lines, pattern color=red]    (1.68, -0.39)--(1.68, -0.3) -- (1.72, -0.3)--(1.72,-0.39)-- cycle; 
\draw[pattern=north west lines, pattern color=orange]     (0.6, -0.65)--(0.6,  -0.44) --  (0.55, -0.44)--(0.55,  -0.65)-- cycle; 
\draw[pattern=north east lines, pattern color=red]          (0.6, -0.72)--(0.6,  -0.65) --  (0.55, -0.65)--(0.55,  -0.72)-- cycle; 
\draw [decorate, ultra thick,
    decoration = {calligraphic brace,mirror,amplitude=6pt}] (0.63,-0.73) --  (0.63,-0.45);
\node at (0.8, -0.6){$\mathcal T_{\alpha}^{\,2}(\mathcal C\,)$};    
%pointers and labels to deletions
\node at (1.9, 0.1){$\mathcal D_1 = $};
\node at (1.73, 0)[pin={[pin edge=<-, pin distance=12pt]0:{$\mathcal T_{\alpha}(\mathcal C_1\,)$}}] {};
\node at (1.73,  -0.35)[pin={[pin edge=<-, pin distance=12pt]0:{$\mathcal T_{\alpha}(\mathcal C_2\,)$}}] {};
%add  1 + 1x2;  
\draw[pattern=north west lines, pattern color=blue]    (-0.28, -0.72)--(-0.28, 0.2) -- (-0.3,0.2)--(-0.3,-0.72)-- cycle; 
\draw[pattern=north east lines, pattern color=green]    (0.25,0.2)--(0.22, 0.2) -- (0.22, 0.28)--(0.25,0.28)-- cycle; 
\draw[pattern=north west lines, pattern color=blue]    (0.25,0.35)--(0.22, 0.35) -- (0.22, 0.28)--(0.25,0.28)-- cycle; 
%pointers and labels to additions
\node at (-.3,  -0.3)[pin={[pin edge=<-, pin distance=12pt]180:{$\mathcal T_{\alpha'}(\mathcal C_1)$}}] {};
\node at (-.55, -0.2){$\mathcal T_{A^{-1}}(\mathcal D_1) = $};
\node at (0.22,  0.22)[pin={[pin edge=<-, pin distance=40pt]165:{$\mathcal T_{\alpha'}(\mathcal C_2\,)$}}] {};
\node at (0.22,  0.32)[pin={[pin edge=<-, pin distance=40pt]130:{$\mathcal T_{\alpha'}^{\,2}(\mathcal C_1\,)$}}] {};
%mark C_2
\fill [opacity= 0.15, gray]  (0.72,0.35)--(0.78, 0.35) -- (0.78, -0.44)--(0.72,-0.44)-- cycle; 
\node at (.75, 0.38)[pin={[pin edge=<-, pin distance=12pt]90:{$\mathcal C_2$}}] {};
\end{tikzpicture}  
}
\caption{{\bf Quilting for close neighbors, large $\alpha$.}  Quilting from  $\Omega_{3, 0.86}$ to $\Omega_{3, 0.855}$ (not fully to scale). Blocks $\mathcal B_{i,j}$, also denoted by $(i,j)$.   The forward $\mathcal T_{\alpha}$-orbit of $\mathcal C  = \mathcal C_1 \cup \mathcal C_2$ is deleted, while the forward  $\mathcal T_{\alpha'}$-orbit of $\mathcal C$ is added, until synchronization causes a ``hole" excised due to the first of these, but to be ``patched" due to the second (not shown here, but compare with Figure~\ref{f:smallAlpQuilt}\,).   See Proposition~\ref{p:closeNbhsNotJustShiftChangesQuiltTogether}. }%
\label{f:quiltLargeAlpLessThanDelta}%
\end{figure}
%-------------------------------------------------------------------------------------------
%------------------------------------------------------------------------ 
 \begin{Prop}\label{p:closeNbhsNotJustShiftChangesQuiltTogether}  Suppose that $\mathcal F = \{ T_{\alpha} \,\vert\, \alpha \in \mathcal I\}$ is a  family of piecewise M\"obius interval maps and $\alpha' <\alpha$ are close neighbors with $\alpha$  typical for their common matching interval.  Fix $M$ such that $T_{\alpha}( \, \ell_0(\alpha)\,) = M\cdot \ell_0(\alpha)$, and suppose further that  $\Delta_{T_{\alpha}, T_{\alpha'}}$ is the union of its subsets%$ = \Delta'_{T_{\alpha} ,T_{\alpha'}}  \cup \Delta''_{T_{\alpha}, T_{\alpha'}}$ where 
 \[
 \begin{aligned}
    \Delta_1 &= \{x\,\vert\,  T_{\alpha'}(x) = S^{-1} \cdot T_{\alpha}(x)\,\}\;\;\text{ and}\\ 
    \Delta_2 &= \{x\,\vert\,  T_{\alpha'}(x) = M S^{-1} \cdot T_{\alpha}(x)\,\}.
 \end{aligned}
 \]   
 Then $\Omega_{\alpha'}$ can be finitely quilted from $\Omega_{\alpha}$.    Furthermore,  $\Omega_{\alpha}$ has locally constant fibers if and only if $\Omega_{\alpha'}$  does. 
\end{Prop}
%-----------------------------------------------------------------------
\begin{proof}  Let $\mathcal C_1, \mathcal C_2 \subset \Omega_{\alpha}$ be the sets projecting to $\Delta_1$ and $\Delta_2$, respectively.   The proof of Proposition~\ref{p:closeNbhsQuiltTogether}  shows that the main interest here is understanding  the $\mathcal T_{\alpha}$- and $\mathcal T_{\alpha'}$-orbits of $\mathcal C_2$.   

As in the previous proof, let the matching interval of $\alpha, \alpha'$  be $J$, and  let   $m, n$ be its matching exponents. 
We also again let $U = U_{\alpha}$ be the M\"obius transformation such that $\ell_m(\alpha) = U L_J\cdot  \ell_0(\alpha)$ and $V = V_{\alpha}$ be such that $r_n(\alpha) = V R_J\cdot r_0(\alpha)$. 

We have that $\mathcal T_{\alpha}(\mathcal C_1 \cup \mathcal C_2)$ is that part of $\Omega_{\alpha}$ fibering over $[r'_0, r_0]$ and hence   $\mathcal T_{\alpha'}(\mathcal C_2)$ and $\mathcal T_{\alpha'}^{2}(\mathcal C_1)$ are given by applying $\mathcal T_M$ to non-intersecting subsets of the plane fibering over $[\ell'_0, \ell_0]$.   The previous proof applies to show that  $\mathcal T_{\alpha'}^{m+1}(\mathcal C_1) = \mathcal T_{U  L_J S^{-1}} \circ \mathcal T_{\alpha}(\mathcal C_1)  = \mathcal T_{V R_J} \circ \mathcal T_{\alpha}(\mathcal C_1) =  \mathcal T_{\alpha}^{n+1}(\mathcal C_1)$.  Since the $\mathcal T_{\alpha}(\mathcal C_1)$ and $\mathcal T_{\alpha}(\mathcal C_2)$  fiber over $[r'_0, r_0]$, their initial $\mathcal T_{\alpha}$-orbits are given by the same sequence of M\"obius transformations.  Since $\alpha'$ and $\alpha$ are close neighbors, we have in fact that $\mathcal T_{\alpha}^{n+1}(\mathcal C_2) = \mathcal T_{V R_J} \circ \mathcal T_{\alpha}(\mathcal C_2)$.  Since $\mathcal T_{\alpha'}(\mathcal C_2) = \mathcal T_{MS^{-1}}\mathcal T_{\alpha}(\mathcal C_2)$, we conclude that $\mathcal T_{\alpha}^{n+1}(\mathcal C_2) = \mathcal T_{\alpha'}^{m}(\mathcal C_2)$.  The disjointness of the initial $\mathcal T_{\alpha'}$-orbits of $\mathcal C_1$ and $\mathcal C_2$ follows from the disjointness of $\mathcal T_{\alpha}(\mathcal C_1)$  and $\mathcal T_{\alpha}(\mathcal C_2)$.  The disjointness of each of these orbits up to the $m^{\text{th}}$ and $m-1^{\text{st}}$ step is argued as in the previous proof, as is the disjointness of these initial orbits with $\Omega_{\alpha}$.   
\end{proof}

\section{Application: An alternate path to proving properties of  Nakada's $\alpha$-continued fractions}\label{s:AltNakErg}

 We now give a rather technical  application of our methods.  We use an alternate description of the planar extension for each $T_{\alpha}$ given in \cite{KraaikampSchmidtSteiner} and certain   results of \cite{KraaikampSchmidtSteiner}  (not relying on the ergodicity of the $T_{\alpha}$) about the planar extensions of  Nakada's $\alpha$-continued fractions,  to recover the following result.
%------------------------------------------------------------------------ 
 \begin{Thm}\label{t:AllNakErg}[Luzzi-Marmi  2008, \cite{LM}]  For every $0<\alpha\le 1$,  the Nakada $\alpha$-continued fraction is ergodic.   
 \end{Thm}
%------------------------------------------------------------------------ 
In fact,  we rely on Theorem~\ref{t:evenExpanErgoNaturally} and thus  find more than just ergodicity.    In particular, we find the following.

%------------------------------------------------------------------------ 
 \begin{Thm}\label{t:quasiModo} For every $0<\alpha\le 1$,  the dynamical systems of $\alpha, \alpha'$ are quasi-isomorphic if and only if they have the same entropy value.   
 \end{Thm}
%------------------------------------------------------------------------ 

%\bigskip 
\subsection{Review of notation and results of \cite{KraaikampSchmidtSteiner} }  Recall from the introduction and \S~\ref{ssHitoshi} that both  \cite{CT}  and \cite{KraaikampSchmidtSteiner} proved the continuity of the entropy function $\alpha \mapsto h(T_{\alpha})$ for Nakada's $\alpha$-continued fractions, when $0<\alpha\le 1$.  The second group of authors created planar extensions for the $T_{\alpha}$   in the form of $\Omega_{\alpha} =  \overline{\big\{\mathcal{T}_\alpha^n(x,0) \mid x \in [{\alpha-1},\alpha),\, n \ge 0\big\}}$, showed the continuity of the $\mu$-mass of these, and argued using Abramov's formula \eqref{e:AbramForm} to reach the continuity result.  All of this built upon the earlier result of Luzzi-Marmi \cite{LM} that each of the interval maps is ergodic with respect to the appropriate measure.   

Here we wish to show that the techniques of this paper can be used to proof the ergodicity of the interval maps.   In brief,  whereas \cite{KraaikampSchmidtSteiner} prove by arguments based upon the ergodicity of $T_{\alpha}$ that   $\Lambda_{\alpha}$ given in \eqref{e:lambda} below is also a valid  expression  for the planar extension of $T_{\alpha}$,  we turn this around and rather prove that $\Lambda_{\alpha}$ gives a planar extension and then apply   Theorem~\ref{t:evenExpanErgoNaturally} to deduce that $T_{\alpha}$ is ergodic. 

We now briefly review some notation and arguments  from \cite{KraaikampSchmidtSteiner}.

  Recall from  \S~\ref{ssHitoshi} that for $\varepsilon \in \{-1,1\}$ and $d \in \mathbb N$, we have $M_{(\varepsilon:d)} = \begin{pmatrix} -d& \varepsilon\\1&0  \end{pmatrix}$ and $N_{(\varepsilon:d)} = \begin{pmatrix} 0& 1\\ \varepsilon &d  \end{pmatrix}$.     
 It is easily verified both that exactly the  digits $(+1:d)$ and $(-1:d+1)$ are such that the image of the open interval $(0,1)$ under $N_{(\varepsilon:d)}$ meets the open interval $(1/(d+1), 1/d)$, and that  for these two digits we have $N_{(\varepsilon:d)} \cdot [0,1] = [\,1/(d+1), 1/d\,]$.     In particular, for each $d$ and any $y$ we find that 
 \begin{equation}\label{e:yRelationNs} 
  N_{(+1:d)} \cdot y =  N_{(-1:d+1)}\cdot (1-y). 
  \end{equation}
Equivalently with $W$ is as in \eqref{e:wmat},  $N_{(+1:d)} \cdot y =  N_{(-1:d+1)}W^t\cdot y$.    Note that since $W$ is of projective order two,  this accords with the easily verified identify:  $M_{(+1:d)} = M_{(-1:d+1)} W$.
  
 For $\alpha \in (0,1]$, we let   $d_\alpha(\alpha)$ be the first $\alpha$-digit of $r_0(\alpha) = \alpha$ and define
\[ \mathscr A_{\alpha} = \{\,(-1:d') \mid 2 \le d' \le d_\alpha(\alpha)+1\} \cup \{(+1:d_\alpha(\alpha))\}.\]

 The approach of \cite{KraaikampSchmidtSteiner} is to list the matching intervals of parameter $\alpha$ by way of certain words $v$, the details of which are not necessary for the current application.   
   For each $v$, \cite{KraaikampSchmidtSteiner} shows that the matching interval indexed by $v$ contains a unique  `atypical' value, this is $\alpha = \chi_v$ which is identified by  the $T_{\chi_v}$-orbits of the endpoints of $\mathbb I_{\chi_v}$ both reaching $x=0$ one step `earlier' than for the matching for the typical values in this matching interval.

Recall that $\mathcal E$ is the complement in $(0,1]$ of the union of the matching intervals of $\alpha$.   For  $\alpha \in \mathcal E$ or $\alpha =\chi_v$  for some $v$,  let  $\mathscr{L}'_{\alpha}$ be the words in 
$\mathscr A_{\alpha}$ which are  admissible $\alpha$-expansions (as well as the empty word).    Then (\cite{KraaikampSchmidtSteiner}, Lemma~7.11) shows that 
  the region, which we rename for clarity's sake, 
\begin{equation}\label{e:lambda}
 \Lambda_{\alpha} = \overline{\bigcup_{w \in \mathscr{L}'_\alpha} T_{\alpha}^{\vphantom{I}|w|}(\Delta_{\alpha}(w)\,) \times N_w \cdot \big[0, \tfrac{1}{d_\alpha(\alpha)+1}\big]}
 \end{equation}
is a bijectivity domain for $\mathcal T_{\alpha}$.  In fact, the lemma is stated for all  $\alpha$, upon making minor adjustments for the remaining $\alpha$:    Each such   `remaining' $\alpha$  is in the same matching interval as $\chi_v$ for some $v$, and one defines  $\mathscr{L}'_{\alpha} =  \mathscr{L}'_{\chi_v}$ and 
replaces $T_{\alpha}^{\vphantom{I}|w|}(\Delta_{\alpha}(w)\,)$ by the $J^\alpha_{\vphantom{I}w}$ of  (\cite{KraaikampSchmidtSteiner}, (7.2) ) --- this last is only a change in the cases that $w$ has a suffix which consists of a prefix of the digits of the $\alpha$-expansion of either $\alpha$ or $\ell_0(\alpha) = \alpha-1$ extending beyond where matching occurs (in a sense, the adjustment is to  keep  the digits up to one step before matching).    To repeat, their proof (in all cases) relies in part on the ergodicity of the $T_{\alpha}$ and involves showing that the bijectivity domain $\Omega_{\alpha}$  is equal to what we have denoted $\Lambda_{\alpha}$.       We now turn this around, and  for $\alpha \in (0,1)$ begin with $\Lambda_{\alpha}$ to show ergodicity of $T_{\alpha}$ and more.

\subsection{Proving $\Lambda_{\alpha}$ is a bijectivity domain to  conclude $T_{\alpha}$ is ergodic}   
    We  aim to show that $\mathcal T_{\alpha}$ is bijective on $\Lambda_{\alpha}$ (as always, here and throughout, up to $\mu$-measure zero sets).  
    
\subsubsection{Surjectivity implies injectivity}    Since a M\"obius transformation is identified by its values on three points, and each $\mathcal T_{(\varepsilon:d)}$ is (locally) measure preserving,  that surjectivity implies  injectivity can be argued as in (\cite{KraaikampSchmidtSteiner}, Lemma~5.2):   In brief, $\Lambda_{\alpha}$ can be partitioned by blocks $\mathcal D_a$, each projecting to its cylinder indexed by $a$, upon each of which $\mathcal T_{\alpha}$ is injective and measure preserving;  the sum over the various $a$ of the $\mu(\mathcal T_{\alpha}(\mathcal D_a))$ hence equals $\mu(\Lambda_{\alpha})$,  which by the surjectivity equals  $\mu(\mathcal T_{\alpha}(\Lambda_{\alpha}))$, but this in turn  equals the measure of the union of the $\mathcal T_{\alpha}(\mathcal D_a)$.  Since the sum of the $\mu(\mathcal T_{\alpha}(D_a))$ equals the measure of the union of the $\mathcal T_{\alpha}(\mathcal D_a)$, injectivity holds up to measure zero.

  Certainly the image of $\Lambda_{\alpha}$
under $\mathcal T_{\alpha}$ contains the union over the non-empty words $w \in \mathscr{L}'_{\alpha}$ of the $T_{\alpha}^{\vphantom{I}|w|}(\Delta_{\alpha}(w)\,) \times N_w \cdot \big[0, \tfrac{1}{d_\alpha(\alpha)+1}\big]$.     The main challenge is to show that all of $\mathbb I_{\alpha} \times \big[0, \tfrac{1}{d_\alpha(\alpha)+1}\big]$ is in the image.  For this, we introduce notation for the fiber in $\Lambda_{\alpha}$ over a point $x$:  for each $x \in \mathbb I_{\alpha}$, let $\Phi_{\alpha}(x) = \{y \mid (x,y)\in \Lambda_{\alpha}\}$.

\subsubsection{Surjectivity follows from fiber symmetry}\label{sss:RectangleAndFiberComplements}   We show surjectivity of $\mathcal T_{\alpha}$ 
by way of an interesting detail that seems not to have been observed in the literature.     The fibers over the cylinders of the values not in $\mathscr A_{\alpha}$ satisfy a certain symmetry property.   Note that the matrix  $W$ from \eqref{e:wmat} acts as $x \mapsto -1/(x+1)$ while its transpose acts by $W^t \cdot y = 1-y$.  We will show that  for any   sufficiently large negative $x$,  the sets  $W^t \cdot \Phi_{\alpha}(x)$ and $\Phi_{\alpha}(W\cdot x)$ are disjoint and have union whose closure is $[0,1]$.   The reader is encouraged to view the various representations of planar domains $\Omega_{\alpha}$ given in, say,  \cite{KraaikampSchmidtSteiner} to see that this is reasonable.    See also Figure~\ref{f:NakadaCFpt39NatExt}.

% constants for figure fo r alpha = 0.39
\def\litG{0.618034}%
\def\lGSqrd{0.381966}%
\def\lGoverSqrtFive{0.276393}%
%-----------------------------Figure Planar Ext alpha=0.39 -------------------------------
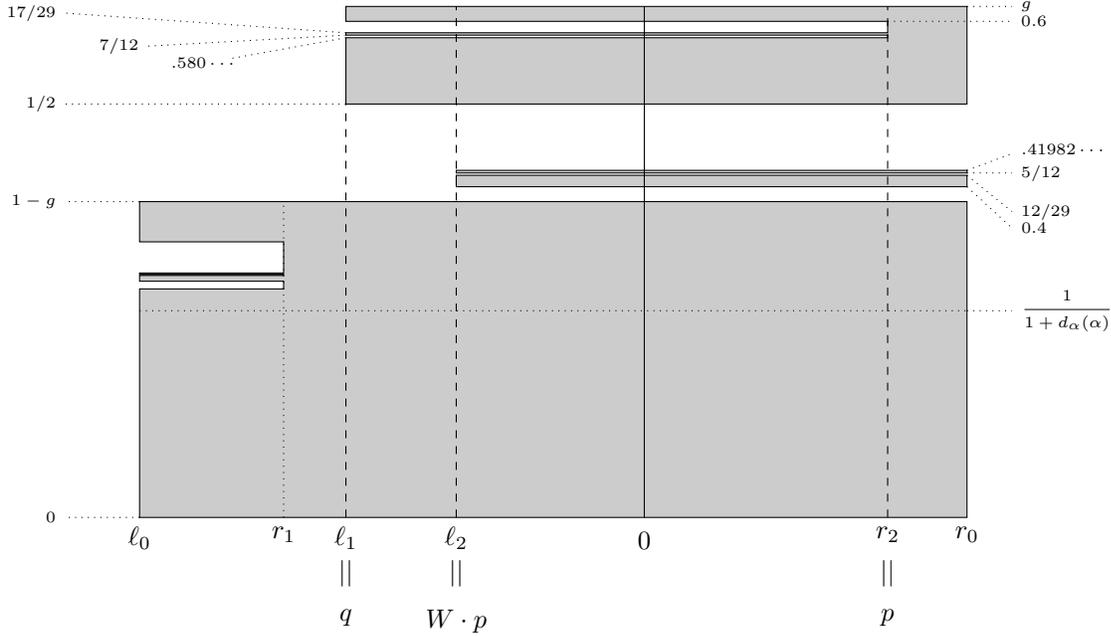
\begin{figure}[ht]
\begin{tikzpicture}[scale=11,fill=black!20]
%\small
% lowest connected piece 
\filldraw(-0.61,0.0)--(-0.61,\lGoverSqrtFive)--(-17/39,\lGoverSqrtFive)--
(-17/39,2/7)--(-0.61,2/7)--(-0.61,12/41)--(-17/39,12/41)--(-17/39,5/17)--
(-0.61,5/17)--(-0.61,0.295686)--(-17/39,0.295686)--(-17/39,1/3)--(-0.61,1/3)
--(-0.61,\lGSqrd)--(0.39,\lGSqrd)--(0.39,0.0)--cycle;
%2nd lowest piece, is rectangle
\filldraw(-5/22,0.4)--(-5/22,12/29)--(0.39,12/29)--(0.39,0.4)--cycle;
\draw[thin,dotted](.39,0.4)--(.4444,0.35) node[below, right]{\tiny{$0.4$}};
\draw[thin,dotted](.39,12/29)--(.4444,0.37);
\node[right]at(.4444,0.37) {\tiny{$12/29$}};
%3rd lowest piece, is rectangle
\filldraw(-5/22,5/12)--(-5/22,0.41982)--(0.39,0.41982)--(0.39,5/12)--cycle;
\draw[thin,dotted](.39,5/12)--(.4444, 5/12);
\node[below, right] at(.4444, 5/12) {\tiny{$5/12$}};
\draw[thin,dotted](.39,0.41982)--(.4444, 0.44);
\node[right] at(.4444, 0.445) {\tiny{$.41982\cdots$}};
%\draw[thin,dotted](.39,12/29)--(.4444,0.364) node[right]{$12/29$};
%4th lowest = highest piece 
\filldraw(-0.36066,0.5)--(-0.36066,0.580179)--(5/17,0.580179)--(5/17,7/12)--(-0.36066,7/12)--(-0.36066,17/29)--(5/17,17/29)--(5/17, 0.6)--(-0.36066,0.6)--(-0.36066,\litG)--(0.39,\litG)-- (0.39,0.5)--cycle;
\draw[thin,dotted](-0.36066,.5)--(-0.7,.5)node[left]{\tiny{$1/2$}};
\draw[thin,dotted](-0.36066,0.5802)--(-0.5,0.555);
\node[left] at (-0.48,0.55) {\tiny{$.580\cdots$}};
\draw[thin,dotted](-0.36066,7/12)--(-0.6,.57) node[left]{\tiny{$7/12$}};
\draw[thin,dotted](-0.36066,17/29)--(-0.7,0.61) node[left]{\tiny{$17/29$}};
\draw[thin,dotted](5/17,0.6)--(.4444,0.6) node[right]{\tiny{$0.6$}};
\draw[thin,dotted](.39,\litG)--(.4444,\litG) node[right]{\tiny{$g$}};
\draw (.0,.618)--(0,0) node[below=.5ex]{$0$};
\node[below] at (-.61,.0) {$\ell_0$};
\node[below] at (-5/22,.0) {$\ell_2$};
\node[below] at (-5/22,-0.04) {$||$};
\node[below] at (-5/22,-0.1) {$W\cdot p$};
\draw[thin,dashed](-5/22,.0)--(-5/22,.618);
\node[below] at (0.39,.0) {$r_0$};
\node[below] at (5/17,.0) {$r_2$};
\node[below] at (5/17,-0.04) {$||$};
\node[below] at (5/17,-0.1) {$p$};
\draw[thin,dashed](5/17,.0)--(5/17,\litG);
\node[below] at (-17/39,0) {$r_1$};
\draw[dotted](-17/39,.0)--(-17/39,0.381966);
\node[below] at (-0.36066,.0) {$\ell_1$};
\node[below] at (-0.36066,-.04) {$||$};
\node[below] at (-0.36066,-0.1) {$q$};
\draw[dashed](-0.36066,.0)--(-0.36066,0.5);
\draw[thin,dotted](-.61,0)--(-.7,0) node[left]{\tiny{$0$}};
%\draw[thin,dotted](-.5556,.3333)--(.4444,.3333)node[right]{$1/3$};
\draw[thin,dotted](-.61,0.381966)--(-.7,0.381966)node[left]{\tiny{$1-g$}};
%promised rectangle 
%\draw[fill= blue!20](0,0)--(0,1/4)--(1/4,1/4)--(1/4,0)--cycle;
\draw[dotted] (-0.61, 1/4)--(0.39, 1/4);
\draw[thin,dotted](.39,1/4)--(.4444,1/4) node[right]{\tiny{$\dfrac{1}{1+d_{\alpha}(\alpha)}$}};
\end{tikzpicture}
\caption{The planar domain $\Omega_{\alpha}$ for Nakada's continued fraction of $\alpha = 0.39$, see Figure~\ref{f:nakadaCfPt39}  for the graph of $T_{\alpha}$.   Marked $x$-values include: $r_i = T_{\alpha}^i(\alpha), \; \ell_i =  T_{\alpha}^i(\alpha-1), 0\le i\le 2$.  Here $d_{\alpha}(\alpha) = 3$.   The $y$-fibers $\Phi(x)$ are constant for $x$ between $W\cdot p$ and $p$, and $[0, 1]$ is the closure of   $\Phi(x) \sqcup W^t\cdot \Phi(x)$ for each such $x$. The proof of Proposition~\ref{l:surjectivity}  refers to $p, W\cdot p, q$ in general cases.}
\label{f:NakadaCFpt39NatExt}
\end{figure}
%-----------------------------------------------------------------------

%------------------------------------------------------------------------ 
 \begin{Prop}\label{l:surjectivity}      For $\alpha \in (0,1)$ the map $\mathcal T_{\alpha}$ is bijective from $\Lambda_{\alpha}$ to itself. 
 \end{Prop}
%-----------------------------------------------------------------------
\begin{proof}  Fix $\alpha$.    
From the definition of $\Lambda_{\alpha}$,  surjectivity onto the complement of   $\mathbb I_{\alpha} \times \big[0, \tfrac{1}{d_\alpha(\alpha)+1}\big]$ is immediate.  

 The proof of   (\cite{KraaikampSchmidtSteiner}, Lemma~5.1) shows,   based upon the fact that there is an explicit manner to rewrite $T_{\alpha}$-orbits in terms of regular continued fraction $T_1$-orbits,  that the rectangle  $\big[0, \tfrac{1}{d_\alpha(\alpha)+1}\big]^2$ is  contained in the closure of the $\mathcal T_{\alpha}$-orbits of the points contained in this rectangle.    The admissible $(\varepsilon:d) \notin \mathscr A_{\alpha}$ are exactly those values such that  $N_{(\varepsilon:d)}\cdot [0,1] \subset \big[0, \tfrac{1}{d_\alpha(\alpha)+1}\big]$.        Hence, the $y$-values here show that each $\mathcal T_{\alpha}$-orbit returns to the rectangle only upon an application of some $\mathcal T_{(\varepsilon:d)}$ with $(\varepsilon:d) \notin \mathscr A_{\alpha}$.

  It follows that  for each  $d \ge d_\alpha(\alpha)+1$, we have that $[1/(d+1), 1/d]$ equals the closure of the union of $N_{(+1:d)}\cdot \Phi_{\alpha}(x)$  with $N_{(-1:d+1)}\cdot \Phi_{\alpha}(x')$   whenever  $x \in \Delta_{\alpha}(+1:d)$ and $x' \in \Delta_{\alpha}(-1:d+1)$    are  such that $T_{\alpha}$ sends them to the same value in  $\big[0, \tfrac{1}{d_\alpha(\alpha)+1}\big]$.    By \eqref{e:yRelationNs} (in the equivalent form given in the line directly below it), this implies that $[0,1] =   \overline{W^t\cdot \Phi_{\alpha}(x)  \cup \Phi_{\alpha}(x')}$ for each such pair.

Now, for $\alpha \in \mathcal E$  (\cite{KraaikampSchmidtSteiner}, Lemma~7.9) shows that all of the digits of the expansions of  both $\alpha-1$ and $\alpha$ are contained in $\mathscr A_{\alpha}$.   Since the only non-full cylinders are associated with prefixes of these expansions,  the fibers  $\Phi_{\alpha}(x)$ are constant for all $x \in \big[\tfrac{-1}{d_\alpha(\alpha)+1}, \tfrac{1}{d_\alpha(\alpha)+1}\big]$.   Therefore,   $[0,1] =   \overline{W^t\cdot \Phi_{\alpha}(x)  \cup \Phi_{\alpha}(x')}$ holds for {\em every} pair $x \in \Delta_{\alpha}(+1:d)$ and $x' \in \Delta_{\alpha}(-1:d+1)$  that are sent by $T_{\alpha}$ to the same value.   That is,   the closure of the union of the images of $\Lambda_{\alpha}$ under the various $\mathcal T_{(+1:d)}$ and  $\mathcal T_{(-1:d+1)}$ fills out all of $\mathbb I_{\alpha}\times \big[0, \tfrac{1}{d_\alpha(\alpha)+1}\big]$.   Surjectivity holds in this case. 

   In the case of $\alpha$ of the form $\chi_v$,   (\cite{KraaikampSchmidtSteiner}, Lemma~7.9) shows that the digits of the expansions of the two endpoints $\alpha-1, \alpha$ remain in $\mathscr A_{\alpha}$ until they match at the value zero.  One finds that the fibers $\Phi_{\alpha}(x)$ are constant for all $x>0$ and also for all $x<0$ whose $\alpha$-digit is at least $(-1:  d_{\alpha}(\alpha)+2)$.   We conclude also in this case that  $[0,1] =   \overline{W^t\cdot \Phi_{\alpha}(x)  \cup \Phi_{\alpha}(x')}$ holds for {\em every} pair $x,x'$ as above, and again the result holds. \\ 
   
    In the remaining case,  $\alpha$ is in the same matching interval as some $\chi_v$,  and   (\cite{KraaikampSchmidtSteiner}, Lemma~7.9) shows that up to their penultimate digits before matching, digits of the expansions of the two endpoints $\alpha-1, \alpha$ remain in $\mathscr A_{\alpha}$; (\cite{KraaikampSchmidtSteiner}, Lemma~6.2) implies that the values exactly before matching differ by an application of $W$.  Let $p$ denote the larger of these values, thus $W\cdot p$ is the other value; also let $q$  be the maximum value of the remainder of $T_{\alpha}$-orbits of the endpoints $\alpha-1, \alpha$ up to these index values.  The fibers $\Phi_{\alpha}(x)$ are  constant over each of the intervals  $[q, W\cdot p),   [W\cdot p, p), [p, \alpha)$,  with respective values  $\Phi_{\alpha}(q), \Phi_{\alpha}(W\cdot p), \Phi_{\alpha}(p)$.  Directly related to this is that for any $d>0$,  points $x \in \Delta_{\alpha}(+1:d)$ and $x' \in \Delta_{\alpha}(-1:d+1)$  are sent by $T_{\alpha}$ to the same value if and only if  $x' = W\cdot x$.  (We could have used this in the previous cases, but preferred to minimize notation.)

Since $p>0$, there exists some $d>d_{\alpha}(\alpha)$ such that $p$ is strictly greater than all values in $\Delta_{\alpha}(+1:d)$.    Hence   for all   $x \in \Delta_{\alpha}(+1:d)$ we have  $\Phi_{\alpha}(x) = \Phi_{\alpha}(W\cdot p)$ and since also $W\cdot p < W\cdot x$ also $\Phi_{\alpha}(W\cdot x) = \Phi_{\alpha}(W\cdot p)$.  
Since there are $x \in \Delta_{\alpha}(+1:d)$ such that $[0,1] =   \overline{\Phi_{\alpha}(x)  \cup W^t\cdot \Phi_{\alpha}(W\cdot x)}$, we find that it is always the case that 
 \begin{equation}\label{e:midFiberSymm} 
 [0,1] =   \overline{\Phi_{\alpha}(W\cdot p)  \cup W^t\cdot  \Phi_{\alpha}(W\cdot p)}.
\end{equation} 
From this for  any $0<x<p$ we find that $[0,1] =   \overline{\Phi_{\alpha}(x)  \cup W^t\cdot  \Phi_{\alpha}(W\cdot x)}$ and in particular for all $d>0$ such that $p$ is strictly greater than all values in $\Delta_{\alpha}(+1:d)$ we have that 
$\mathbb I_{\alpha} \times [1/d, 1/(d+1)]$ is contained in $\mathcal T_{\alpha}(\Lambda_{\alpha})$. 

We next claim that 
\[\Phi_{\alpha}(p)  \cup W^t\cdot  \Phi_{\alpha}(q) = \Phi_{\alpha}(W\cdot p)  \cup W^t\cdot  \Phi_{\alpha}(W\cdot p).\] 
To prove this, recall that  \cite{KraaikampSchmidtSteiner} use $k',k$ as matching exponents and use $E$ to denote the matrix which acts as shift by $-1$, and show that there is an $M_v$ such that $T_{\alpha}^{k'-1}(\alpha-1) = M_v\cdot (\alpha-1) = M_v E \cdot \alpha$   and $T_{\alpha}^{k-1}(\alpha) = W M_v E\cdot \alpha$.   Thus $\{p, W\cdot p\} =  \{M_v E \cdot \alpha, WM_v E \cdot \alpha \}$ and $\Phi_{\alpha}(W\cdot p)$ is the union of $\Phi_{\alpha}(q)$ with one of either $N_v\cdot \Phi_{\alpha}(\alpha-1)$ or $W^t N_v (E^{-1})^t \cdot \Phi_{\alpha}(\alpha)$.  Similarly,  $\Phi_{\alpha}(p)$ is the union of $\Phi_{\alpha}(q)$ with both   $N_v\cdot \Phi_{\alpha}(\alpha-1)$ and  $W^t N_v (E^{-1})^t \cdot \Phi_{\alpha}(\alpha)$.     Since $W$ and hence $W^t$ is of projective order two, the claim holds. 

Using \eqref{e:midFiberSymm} and the claim,  we find for all $d>d_{\alpha}(\alpha)$ that $\mathbb I_{\alpha} \times [1/d, 1/(d+1)]$ is contained in $\mathcal T_{\alpha}(\Lambda_{\alpha})$.  Thus, our proof of surjectivity is complete.    

  For completeness, we   recall that the previous subsubsection outlines a proof showing that injectivity  now follows.
  \end{proof}

 \subsection{Proof of ergodicity}
  
We now complete our proof of Theorem~\ref{t:AllNakErg}.  It is immediate that each $\Lambda_{\alpha}$ has positive $\mu$-measure.    Since   $N_a\cdot [0,1] \subset [0,1]$ for every  possible digit for any given $\alpha$,   it is clear that 
$\Lambda_{\alpha} \subset   [\alpha-1,  \alpha] \times [0,1]$.  Therefore,  the compact $\Lambda_{\alpha}$ is bounded away from $y = -1/x$ and has finite vertical fibers.
Recall that Lemma~\ref{p:NakadaBddRange}  guarantees that every rational $\alpha \in (0,1]$ and every $\alpha \in \mathcal E$ is of bounded non-full range.   Furthermore,  (\cite{KraaikampSchmidtSteiner} Theorem~5) shows that if $\alpha$ is an  endpoint of a matching interval then both  $\alpha-1$ and $\alpha$ have periodic $T_{\alpha}$-expansions and thus these maps also are of bounded non-full range.       By Theorem~\ref{t:evenExpanErgoNaturally} with  $ \Omega =  \Lambda_{\alpha}$, we find ergodicity of $T_{\alpha}$ and the other properties listed in the statement of the theorem.    Each of the remaining $\alpha' \in (0,1)$ lies in the interior of some matching interval and the density of the rationals gives that $\alpha'$ has some typical  $\alpha$ as a close neighbor.   From Proposition~\ref{p:closeNbhsQuiltTogether}, combined with  Theorem~\ref{t:finQuiltIsFine}  we have that each   $\mathcal T_{\alpha}: \Lambda_{\alpha} \to \Lambda_{\alpha}$ gives the natural extension to the system of $T_{\alpha}$, which is in particular ergodic.  
 
\subsection{Proof of quasi-isomorphism class determined by entropy}   
 We now sketch the proof of Theorem~\ref{t:quasiModo}.    We desire to apply the Rychlik result, Theorem~\ref{t:rychlikBernoulli}.        By the ergodicity result,   for each $\alpha \in (0, 1]$ the map $T_{\alpha}$ has a unique invariant probability measure equivalent to Lebesgue.   Furthermore, each is such that every open subset of $\mathbb I_{\alpha}$ contains a full rank $m$ cylinder for some $m \in \mathbb N$; hence all of $\mathbb I_{\alpha}$ is contained in the $T_{\alpha}^m$-image of this open subset.   Each   $M_{(\varepsilon:d)} = \begin{pmatrix} -d& \varepsilon\\1&0  \end{pmatrix}$   defines a function of $x$ whose derivative has absolute value $x^{-2}$.  Thus, $|1/T_{\alpha}'|$  is certainly bounded on all of $\mathbb I_{\alpha}$ for any of the $\alpha$.   By Theorem~\ref{t:rychlikBernoulli}, the natural extension of each $T_{\alpha}$ is Bernoulli.    Therefore, the result holds due to Ornstein's fundamental result,  Theorem~\ref{t:Ornstein}.

\end{document}